\documentclass[11pt]{amsart}
\usepackage{geometry}                % See geometry.pdf to learn the layout
\usepackage[all]{xypic}
\usepackage{graphicx}
\usepackage{amssymb}
\usepackage{array}
\usepackage{epstopdf}

\newtheorem{thm}{Theorem}[section]
\newtheorem*{theorem*}{Theorem}
\newtheorem*{corollary*}{Corollary}
\newtheorem{prop}[thm]{Proposition}
\newtheorem{lem}[thm]{Lemma}
\newtheorem{rem}[thm]{Remark}
\newtheorem{defn}[thm]{Definition}
\newtheorem{cor}[thm]{Corollary}\newtheorem{example}[thm]{Example}
\newcommand{\nc}{\newcommand}
\nc{\Un}{{\mathbb U}({\mathfrak n}^-)}
\nc{\UnZ}{{\mathbb U}({\mathfrak n}^-_{\mathbb{Z}})}
\nc{\UnA}{{\mathbb U}({\mathfrak n}^-_{A})}

\newcommand{\ZZ}{\mathbb{Z}}

\title[Degenerate flag varieties and Schubert varieties]{Degenerate flag varieties and Schubert varieties: \\a characteristic free approach} 
\author{G. Cerulli Irelli, M. Lanini, P. Littelmann}
\address{Giovanni Cerulli Irelli:
Dipartimento di Matematica. Sapienza-Universit\`a di Roma. Piazzale Aldo Moro 5, I-00185, Rome (Italy)}
\email{cerulli@\,mat.uniroma1.it, cerulli.math@googlemail.com}
\address{Martina Lanini: School of Mathematics, University of Edinburgh,
James Clerk Maxwell Building,
Edinburgh, EH9 3FD (UK)}
\email{m.lanini@\,ed.ac.uk}
\address{Peter Littelmann:
Mathematisches Institut
Universit\"at zu K\"oln
Weyertal 86-90
D-50931 K\"oln (Germany)}
\email{peter.littelmann@\,math.uni-koeln.de}

\begin{document}

\begin{abstract}
 We consider the PBW filtrations over $\mathbb Z$ of the irreducible highest weight modules in type ${\tt A}_n$ and ${\tt C}_n$.
 We show that the associated graded modules can be realized as Demazure modules for group schemes of the same
 type and doubled rank. We deduce that the corresponding degenerate flag varieties are isomorphic to Schubert varieties
 in any characteristic.
\end{abstract}
\maketitle

\section*{Introduction}
Introduced by Evgeny Feigin in 2010, degenerate flag varieties naturally arise from a representation theoretic context. In fact, given  a finite dimensional, highest weight irreducible module $V(\lambda)$ for a simple finite dimensional, complex Lie algebra, the corresponding degenerate flag variety ${\mathcal{F}\ell}(\lambda)^a$ is the closure of a certain highest weight orbit in the projectivization of  $V(\lambda)^a$, a degenerate version of $V(\lambda)$. 

If the algebra one starts with is of type ${\tt A}_n$ or ${\tt C}_n$, it was shown in \cite{CL} that, surprisingly, degenerate flag varieties can be realized as Schubert varieties in a partial flag variety of same type and bigger rank. It is hence natural to ask whether also the modules $V(\lambda)^a$ are isomorphic to some already investigated objects. The aim of this paper is to address such a question and provide a positive answer to it. Feigin's degeneration procedure can be carried out over $\mathbb{Z}$ (cf. \cite{FFL3}) and it is in this generality that we decided to approach the problem.

Our main theorem  is the realization of $V(\lambda)^a$ as a Demazure module for a group scheme of same type and doubled rank.  This fact allows us to recover, as a corollary, the above-mentioned realization of ${\mathcal{F}\ell}(\lambda)^a$ as a Schubert variety.  While the arguments in \cite{CL} relied on a linear algebraic description of the degenerate flag variety due to Feigin (cf. \cite{F1}), the proof we obtain here  only uses the definition of  ${\mathcal{F}\ell}(\lambda)^a$  as a closure of an highest weight orbit and it is hence more conceptual.
% and, hopefully, generalizable to other types.

In what follows,  we describe more in detail the main results of this article.

For simplicity let us start with the complex algebraic group $SL_n(\mathbb C)$ and its Lie algebra
$\mathfrak{g}=\mathfrak{sl}_n$. We fix a Cartan decomposition $\mathfrak{g}=\mathfrak{n}^-\oplus \mathfrak{h}\oplus \mathfrak{n}^+$, where
$\mathfrak{n}$ is the subalgebra of strictly upper triangular matrices, $\mathfrak{h}$ is the Cartan subalgebra consisting of diagonal matrices
and $\mathfrak{n}^-$  is the subalgebra of strictly lower triangular matrices. Let $\mathfrak{b}=\mathfrak{h}\oplus \mathfrak{n}^+$ be the
corresponding Borel subalgebra of $\mathfrak{g}$ and let $B\subset G$ be the Borel subgroup with Lie algebra $\mathfrak b$.

We use the notation $\tilde B$, $\tilde{\mathfrak{b}}$, $\tilde{\mathfrak{n}}^+$, $\tilde{\mathfrak{h}}$ and $\tilde{\mathfrak{n}}^-$
for the corresponding subgroup of $\tilde G=SL_{2n}(\mathbb C)$ and subalgebras of $\tilde{\mathfrak{g}}=\mathfrak{sl}_{2n}$.
Let $\mathfrak{n}^{-,a}\subset \mathfrak{sl}_{2n}$ respectively $N^{-,a} \subseteq SL_{2n}(\mathbb C)$
be the following commutative Lie subalgebra respectively commutative unipotent subgroup:
\begin{equation}\label{nilpotentshift}
\mathfrak{n}^{-,a}:=\left\{\left(\begin{array}{cc} 0  & N \\ 0 & 0\hskip-2pt
\end{array}\right)\in \mathfrak{sl}_{2n}\mid N\in \mathfrak{n}^-\right\},
N^{-,a}:=\left\{\left(\begin{array}{cc} 1\hskip-2pt \text{I} & N \\ 0 & 1\hskip-2pt
\text{I}\end{array}\right)\in SL_{2n} \mid N\in \mathfrak{n}^-\right\}.
\end{equation}
We view $\mathfrak{n}^{-,a}$ as the abelianization of $\mathfrak{n}^-$,
i.e. we have the canonical vector space isomorphism between the two vector spaces,  but $\mathfrak{n}^{-,a}$ is endowed with the trivial Lie bracket.
The enveloping algebra of $\mathfrak{n}^{-,a}$ is $S^\bullet(\mathfrak{n}^{-,a})$. The embedding
$\mathfrak{n}^{-,a}\hookrightarrow \tilde{\mathfrak{b}}$ induces an embedding
$S^\bullet(\mathfrak{n}^{-,a})\hookrightarrow{U}(\tilde{\mathfrak{b}})$, so
any ${U}(\tilde{\mathfrak{b}})$-module inherits in a natural way the structure of a
$S^\bullet(\mathfrak{n}^{-,a})$-module.

A well investigated class of ${U}(\tilde{\mathfrak{b}})$-modules are the Demazure modules: let $\mu$ be a dominant
integral weight for $\tilde{\mathfrak{g}}$ and  let $\tilde{V}(\mu)$ be the corresponding irreducible representation.
For an element $w$ of the Weyl group $\tilde{W}$ of $\tilde{\mathfrak{g}}$, the weight space $\tilde{V}(\mu)_{w\mu}$ of
weight $w\mu$ is one-dimensional, fix a generator $v_{w\mu}$. Recall that the Demazure submodule $\tilde{V}(\mu)_{w}$ is
by definition the cyclic ${U}(\tilde{\mathfrak{b}})$--module generated by $v_{w\mu}$, i.e.
$\tilde{V}(\mu)_{w}={U}(\tilde{\mathfrak{b}})\cdot v_{w\mu}$, and the Schubert variety $X(w)$ is the closure
of the orbit $\tilde B.[v_{w\mu}]\subseteq \mathbb P(\tilde V(\mu))$.

A special class of $S^\bullet(\mathfrak{n}^{-,a})$-modules has been investigated in \cite{FFL1, FFL2}. Let $\lambda$ be a dominant
integral weight for ${\mathfrak{g}}$, let ${V}(\lambda)$ be the corresponding irreducible representation and
fix a highest weight vector $v_\lambda$. The PBW filtration on ${U}({\mathfrak{n}^-})$ induces
a filtration on the cyclic ${U}({\mathfrak{n}^-})$-module $V(\lambda)={U}({\mathfrak{n}^-}).v_\lambda$,
the associated graded space $V^a(\lambda):=\text{gr\,}V(\lambda)$ becomes a module for the associated graded
algebra $S^\bullet(\mathfrak{n}^-):=\text{gr\,}{U}({\mathfrak{n}^-})\simeq S^\bullet(\mathfrak{n}^{-,a})$.

The action of $\mathfrak{n}^{-,a}$ on $V^a(\lambda)$ can be integrated to an action of $N^{-,a}$, in analogy with the classical case
we call the closure of the orbit $\mathcal{F}^a_\lambda:=\overline{N^{-,a}.[v_\lambda]}\subseteq \mathbb P(V^a(\lambda))$ the degenerate flag variety.

The aim of this article is to connect these two constructions and extend the results in \cite{CL} to an algebraically closed field $k$ of arbitrary characteristic.
In fact, the results hold even over $\mathbb Z$. For simplicity, we formulate them in the introduction for an algebraically closed field $k$.
In the following we consider the case $G=SL_n(k)$ and $\tilde G=SL_{2n}(k)$, respectively $G=Sp_{2m}(k),$ and $\tilde G=Sp_{4m}(k)$, and
we replace the irreducible module of highest weight $\lambda$ by the Weyl module of highest weight $\lambda$, using the same
notation $V(\lambda)$. For the precise description
of the highest weight $\Psi(\lambda)$ see Definitions~\ref{mappsi} and \ref{mappsisymp}, for a description of the Weyl group element $\tau\in \tilde W$ see 
Definitions~\ref{deftau} and \ref{tauSp},
and the construction of the Lie algebra $\mathfrak n^{-,a}$ in the symplectic case can be found in Section~\ref{symplecticCase}.
For a dominant $G$-weight $\lambda$ let $\lambda^*$ be the dual dominant weight, so for the symplectic case we
have $\lambda=\lambda^*$, and in the $SL_n$ case we have  $\lambda^*=\sum_{i=1}^{n-1} m_i\omega_{n-i}$ for $\lambda=\sum_{i=1}^{n-1} m_i\omega_{i}$
(notation as in \cite{B}).

\begin{theorem*}
Let $\lambda$ be a dominant $G$-weight.
\begin{itemize}
\item[{\it i)}]  The Demazure submodule
$\tilde V_{k}(\Psi(\lambda^*))_\tau$ of the
$\tilde G$-module $\tilde V_{k}(\Psi(\lambda^*))$ is,
as an $\mathfrak n^{-,a}$-module, isomorphic to the abelianized
module $V^a(\lambda)$.
\item[{\it ii)}] The Schubert variety
$X(\tau)\subset \mathbb P(\tilde V(\Psi(\lambda^*))_\tau)$ is isomorphic to the
degenerate  flag variety ${\mathcal{F}^a}(\lambda)$, and this isomorphism induces an
$S^\bullet(\mathfrak{n}^{-,a})$-module isomorphism
$$
H^0(X(\tau),\mathcal L_{\Psi(\lambda^*)})\simeq (V^a(\lambda))^*.
$$
\end{itemize}
\end{theorem*}
Using the isomorphism above,  we deduce the defining relations
for $V^a(\lambda)$ from the defining relations of the Demazure module.
Translated back into the language of the abelianized algebras we get the following:
 in the $SL_n$-case let $R^{++}=R^+$ be the set of positive roots, in the symplectic case set
$$
R^{++}=\{\epsilon_i-\epsilon_j\mid 1\leq i< j\leq m\}\cup \{2\epsilon_i\mid
1\leq i \leq m\}.
$$
\begin{corollary*}
The abelianized module $V^a(\lambda)$ is as a cyclic $S^\bullet(\mathfrak n^{-,a})$-module isomorphic to
$S^\bullet(\mathfrak n^{-,a})/I(\lambda)$, where $I(\lambda)$ is the ideal:
$$
I(\lambda)= S^\bullet(\mathfrak n^{-,a})( U(\mathfrak n^+)\circ
\text{span}\{f_\alpha^{(\langle\lambda,\alpha^\vee\rangle+1)};\alpha \in R^{++}\})\subseteq S^\bullet(\mathfrak n^{-,a}).
$$
\end{corollary*}
The identification of the degenerate flag variety as a Schubert variety implies immediately (see also \cite{FF, FFiL}):
\begin{corollary*}
The degenerate flag variety ${\mathcal{F}^a}(\lambda)$ is projectively normal and has rational
singularities.
\end{corollary*}

\section{Some special commutative unipotent subgroups}\label{sectionone}
Let $k$ be a field. Given a subspace $\mathfrak N\subseteq M_n(k)$ and
a vector space automorphism $\eta:\mathfrak N\rightarrow \mathfrak N$,
denote by ${\mathfrak N}_\eta^{a} \subseteq M_{2n}(k)$ respectively
${N}_\eta^{a} \subseteq GL_{2n}(k)$
the following commutative nilpotent Lie subalgebra of $M_{2n}(k)$, respectively
commutative unipotent subgroup of $GL_{2n}(k)$:
$$
{\mathfrak N}_\eta^{a}:=\left\{\left(\begin{array}{cc} 0 & \eta(A) \\ 0 & 0 \end{array}\right)\mid A\in \mathfrak N\right\},\quad
{N}_\eta^{a} :=\left\{\left(\begin{array}{cc} 1\hskip-2pt \text{I} & \eta(A) \\ 0 & 1\hskip-2pt
\text{I}\end{array}\right)\mid A\in \mathfrak N\right\}.
$$
If $\mathfrak N\subseteq M_n(k)$ is a Lie subalgebra, then we think of ${\mathfrak N}_\eta^{a}$
as an abelianized version of $\mathfrak N$. Similarly one may think of ${N}_\eta^{a}$ as
an abelianized version of a subgroup $N\subseteq GL_{2n}(k)$. We will be more precise in the
following examples.

\begin{example}\label{example1}\rm\noindent
Let $k$ be an algebraically closed field of characteristic zero. We fix as a maximal torus $T\subset SL_n$
the subgroup of diagonal matrices, and let $B$ be the Borel subgroup of upper triangular matrices.
Let us denote by $\mathfrak{sl}_n$, $\mathfrak{b}$ and $\mathfrak{h}$ the corresponding Lie algebras
and let $\mathfrak{g}=\mathfrak{n}^-\oplus \mathfrak{h}\oplus \mathfrak{n}^+$ be the Cartan decomposition.
The choice of a maximal torus and a Borel subgroup as above determines the set of
positive roots $\Phi^+$ and hence, according to the adjoint action of
$\mathfrak{h}$, the root space decomposition $\mathfrak{n}^-=\bigoplus_{\alpha\in\Phi^+}
\mathfrak{n}^-_{-\alpha}$.
In this example we set $\mathfrak{N}=\mathfrak{n}^-$, and $N=U^{-}$ is the unipotent radical of the opposite
Borel subgroup $B^-$. The map $\eta$ is the identity map, so we just omit it.
We write henceforth $\mathfrak{n}^{-,a}$ for ${\mathfrak N}^{a}\subset \mathfrak{sl}_{2n}(k)$
and $N^{-,a}$ for $N^{a}\subset SL_{2n}$.

Note that $\mathfrak{n}^{-,a}\subset \mathfrak{sl}_{2n}$ is a Lie subalgebra
of the Borel subalgebra $\tilde{\mathfrak{b}}\subset \mathfrak{sl}_{2n}$
and ${N}^{-,a}$ is an abelian subgroup of the Borel subgroup
$\tilde B \subset SL_{2n}$ (of upper triangular matrices). We can think of
 ${N}^{-,a}$ as an abelianized version of $U^-$.

The subgroup  ${N}^{-,a}$, as well as the Lie algebra $\mathfrak{n}^{-,a}$, is stable under conjugation with respect to the maximal torus
$\tilde T\subset SL_{2n}$, where $\tilde T\subset SL_{2n}$ consists of the diagonal matrices.
The group ${N}^{-,a}$ hence decomposes as a product of root subgroups of the group $SL_{2n}$,
and $\mathfrak{n}^{-,a}$ decomposes into the direct sum of root subspaces for the Lie algebra $\mathfrak{sl}_{2n}$.
We get an induced map $\phi:\Phi^+\rightarrow \tilde\Phi^+$ between the set of positive roots of  $\mathfrak{sl}_{n}$
and the positive roots of  $\mathfrak{sl}_{2n}$, such that $\mathfrak{n}^{-,a}=\bigoplus_{\alpha\in\Phi^+}
\mathfrak{n}^{-,a}_{\phi(\alpha)}$.
\end{example}
\begin{example}\label{example2}\rm\noindent
Let $k$ be an algebraically closed field of characteristic $0$.
Let $\{e_1,\ldots, e_{2n}\}$ be the canonical basis of $k^{2n}$,
we fix a non-degenerate skew symmetric form by the conditions
$\langle e_i,e_j\rangle=\delta_{j,2n-i+1}=-\langle e_{j}, e_i\rangle$ for $1\le i\le n$, $1\le j\le 2n$.
Let $Sp_{2m}$ be the associated symplectic group.
By the choice of the form we can fix as a Borel subgroup $B$ the subgroup
of upper triangular matrices in $Sp_{2m}$, and let $T$ be its
maximal torus consisting of diagonal matrices.
Let us denote by $\mathfrak{sp}_{2m}$, $\mathfrak{b}$ and $\mathfrak{h}$ the
corresponding Lie algebras and let $\mathfrak{g}=\mathfrak{n}^-\oplus \mathfrak{h}\oplus \mathfrak{n}^+$ be the Cartan decomposition.

The choice of the torus and the Borel subgroup as above determines a set of
positive roots $\Phi^+$ and hence, according to the adjoint action of
$\mathfrak{h}$, the root space decomposition $\mathfrak{n}^-=\bigoplus_{\alpha\in\Phi^+}
\mathfrak{n}^-_{-\alpha}$.
In this example we set $\mathfrak{N}=\mathfrak{n}^-$, and $N=U^{-}$ is the unipotent radical of the opposite
Borel subgroup $B^-$.
Let $\eta:\mathfrak{n}^-\rightarrow \mathfrak{n}^-$
be the linear map sending a matrix $(m_{i,j})_{1\le i,j\le 2n}$
to the matrix $(m'_{i,j})_{1\le i,j\le 2n}$, where $m'_{i,j}=m_{i,j}$
if $i\le n$ or $j\le n$ and $m'_{i,j}=-m_{i,j}$ if both indices are strictly larger than $n$.
We write henceforth $\mathfrak{n}_\eta^{-,a}$ for ${\mathfrak N}_\eta^{a}\subset \mathfrak{sp}_{4n}(k)$
and $N_\eta^{-,a}$ for $N_\eta^{a}\subset Sp_{4n}$.

Note that $\mathfrak{n}_\eta^{-,a}\subset \mathfrak{sp}_{4n}$ is a Lie subalgebra
of the Borel subalgebra (of upper triangular matrices) $\tilde{\mathfrak{b}}\subset \mathfrak{sp}_{4n}$
and ${N}_\eta^{-,a}$ is an abelian subgroup of the Borel subgroup
$\tilde B \subset Sp_{4n}$ (of upper triangular matrices). We can think of
 ${N}^{-,a}_\eta$ as an abelianized version of $U^-$.

The subgroup  ${N}^{-,a}_\eta$ is stable under conjugation with respect to the maximal torus
$\tilde T\subset Sp_{4n}$, where $\tilde T\subset Sp_{4n}$ consists of the diagonal matrices.
The group ${N}^{-,a}_\eta$ hence decomposes as a product of root subgroups of the group $Sp_{4n}$
and $\mathfrak{n}^{-,a}$ decomposes into the direct sum of root subspaces for the Lie algebra $\mathfrak{sp}_{4n}$.
We get an induced map $\phi:\Phi^+\rightarrow \tilde\Phi^+$ between the set of positive roots of  $\mathfrak{sp}_{2n}$
and the positive roots of  $\mathfrak{sp}_{4n}$, such that $\mathfrak{n}^{-,a}_{\eta}=\bigoplus_{\alpha\in\Phi^+}
\mathfrak{n}^{-,a}_{\phi(\alpha),\eta}$.

\end{example}
\begin{example}\label{example3}\rm\noindent
To get a characteristic free approach for $G$ as above,
let  $G_{\mathbb Z}$ be a split and connected simple algebraic ${\mathbb Z}$-group of type $\tt A_n$ or $\tt C_n$.
For any commutative ring $A$ set $G_A=(G_{\mathbb Z})_A$, and for a field set $G=G_k$, for this and the following see also \cite{J}.
Then $G_k$ is for any algebraically closed field a reduced $k$-group, and it is connected and reductive.
Its Lie algebra $\text{Lie}(G_{\mathbb Z})$ is a free Lie algebra of finite rank and
$\text{Lie\,} G_k=\text{Lie\,} (G_{\mathbb Z})\otimes_{\mathbb Z} k$. Let $T_{\mathbb{Z}}\subset G_{\mathbb Z}$
be a split maximal torus and set $T_A=(T_{\mathbb Z})_A$ for any ring $A$ and $T=T_k$. We have a root space
decomposition $\text{Lie\,} G= \text{Lie\,} T \oplus \bigoplus_{\alpha\in \Phi} (\text{Lie\,} G)_\alpha$
where $(\text{Lie\,} G)_\alpha=(\text{Lie\,} G_{\mathbb Z})_\alpha\otimes_{\mathbb Z} k$,
and corresponding root subgroups (defined over $\mathbb Z$) $x_\alpha:\mathbb G_a\rightarrow G$
such that the tangent map $\text{d}x_\alpha$ induces an isomorphism between the Lie algebra
of the additive group $\mathbb G_a$ and $(\text{Lie\,} G)_\alpha$. The functor which associates
to any commutative ring $A$ the group $x_\alpha(\mathbb G_a(A))=x_\alpha(A)$ is a closed subgroup
of $G$ denoted by $U_\alpha$, and we have  $\text{Lie\,} (U_\alpha)=(\text{Lie\,} G)_\alpha$.
Over $\mathbb Z$ we denote the corresponding subgroup by $U_{\alpha, \mathbb Z}$,
over a field $k$ we have $U_\alpha = (U_{\alpha, \mathbb Z})_k$.

The construction described above in Examples~\ref{example1} and \ref{example2} makes
(in this language) sense over $\mathbb Z$ or over any field. As before, let $\tilde G$ be the group of the same type but twice the rank,
we denote the corresponding Borel subgroup, maximal torus etc. by $\tilde B$, $\tilde T$ etc.
The construction in the examples above associates to every root $\alpha\in \Phi$ a root $\phi(\alpha)$ in the root system of
$\tilde G$. For the $\mathbb Z$-group $G_{\mathbb Z}$
we have the subgroup $U^-_{\mathbb Z}$ and the Lie algebra
$\mathfrak{n}^-_{\mathbb{Z}}=\bigoplus_{\alpha\in\Phi^+}
\mathfrak{n}^-_{\mathbb{Z},-\alpha}$, and we associate to this pair a new pair
given by a commutative subgroup $N^{-,a}_\eta$ of the $\mathbb Z$-group $\tilde G_{\mathbb Z}$ 
and an abelian Lie algebra $\mathfrak{n}^{-,a}_{\mathbb{Z},\eta}$. The first
is the subgroup of the Borel subgroup $\tilde B_{\mathbb Z}\subset \tilde G_{\mathbb Z}$
generated by the commuting root subgroups $U_{{\phi(\alpha)},\mathbb Z}$, $\alpha\in \Phi^{+}$, and
the second is the abelian Lie algebra $\mathfrak{n}^{-,a}_{\mathbb{Z},\eta}=\bigoplus_{\alpha\in\Phi^+}
\widetilde{(\text{Lie\,} G)}_{\mathbb{Z},\phi(\alpha)}$ given as the sum of root subalgebras.
\end{example}

\section{A special Schubert variety -- the $SL_n$-case}\label{section two}
We want to realize in the situation of Example~\ref{example1} the abelianized
representation $V(\lambda)^{a}$ for $N_{\eta}^{-,a}$
as a Demazure submodule of an irreducible representation for the group
$SL_{2n}$.

\subsection{A special Weyl group element}
Let $\tilde W$ be the Weyl group of $SL_{2n}(\mathbb C)$, it is the symmetric group $\mathcal S_{2n}$
generated by the transpositions $s_i$, $i=1,\ldots,2n-1$.  Let $\mathfrak{h}\subset\mathfrak{g}=\mathfrak{sl}_n$
(respectively $\tilde{\mathfrak{h}}\subset\tilde{\mathfrak{g}}=\mathfrak{sl}_{2n}$) be the Cartan subalgebra of traceless complex
diagonal matrices.  For an element $\alpha\in\mathfrak{h}^\ast$ and an element $h\in\mathfrak{h}$ we denote by $\langle h,\alpha\rangle$ the
evaluation of $\alpha$ in $h$.
Let $\{\epsilon_1,\cdots, \epsilon_{n}\}$ be the elements of the dual vector space $\mathfrak{h}^\ast$ such that $\langle h,\epsilon_i\rangle$
is the $i$-th entry in the diagonal matrix $h\in \mathfrak{h}$.
We use the same notation $\langle\tilde h,\tilde \alpha\rangle$ for elements $\tilde h\in \tilde{\mathfrak{h}}$ and
$\tilde \alpha\in \tilde{\mathfrak{h}}^*$, and the linear forms $\{\tilde{\epsilon}_1,\cdots, \tilde{\epsilon}_{2n}\}$ in $\tilde{\mathfrak{h}}^*$
are defined as above.

The roots of $\mathfrak{g}$ (resp. $\tilde{\mathfrak{g}}$) are the elements
$\alpha_{i,j}:=\epsilon_i-\epsilon_j$ (resp. $\tilde{\alpha}_{i,j}=\tilde{\epsilon}_i-\tilde{\epsilon}_j$) for $i\neq j$. We choose as a
Borel subalgebra of $\mathfrak g$ the subalgebra $\mathfrak{b}$ of upper triangular matrices.
The corresponding simple roots are $\alpha_1,\cdots, \alpha_{n-1}$ given by $\alpha_i:=\alpha_{i,i+1}$.
For every root $\alpha$, we denote by $\alpha^\vee$ its coroot: this is the unique element of $\mathfrak{h}$
such that the reflection $s_\alpha\in\mathfrak{h}^\ast$ along $\alpha$ acts as
$s_\alpha(\lambda)=\lambda-\langle \alpha^\vee,\lambda \rangle\alpha$. Moreover we denote by $E_\alpha$ the corresponding root vector. 
We denote by $\omega_i=\epsilon_1+\cdots+\epsilon_i$ (resp. $\tilde{\omega}_i=\tilde{\epsilon}_1+\cdots+\tilde{\epsilon}_i$)
the i--th fundamental weight of $\mathfrak{g}$ (resp.  $\tilde{\mathfrak{g}}$), where $i=1,2,\cdots, n-1$ (resp. $i=1,\ldots,2n-1$).
They are characterized by the property $\langle \alpha_i^\vee,\omega_j\rangle=\delta_{i,j}$.
\begin{defn}\label{mappsi}
Let $\Psi:\mathfrak{h}^\ast\rightarrow\tilde{\mathfrak{h}}^\ast$ be the linear map defined on the weight lattice as follows
$$
\Psi(\sum_{i=1}^{n-1} a_i\omega_i):=\sum_{i=1}^{n-1} a_i\tilde\omega_{2i}.
$$
\end{defn}
Note that $\Psi$ sends dominant weights to dominant weights.  For every fundamental weight $\tilde{\omega}_k$,
we denote the corresponding parabolic subgroup by $P_{\tilde{\omega}_k}$ and by $\tilde{W}_{\tilde{\omega}_{k}}$
the corresponding subgroup of $\tilde{W}$, which is the Weyl group of the semi-simple part of $P_{\tilde{\omega}_k}$.
Note that $\tilde{W}_{\tilde{\omega}_{k}}$ is generated by all the simple transpositions $s_i$'s but $s_{k}$. Let
$\rho=\omega_1+\cdots+\omega_{n-1}$. Then $\Psi(\rho)=\tilde{\omega}_2+\tilde{\omega}_4+\cdots+\tilde{\omega}_{2n-2}$.
The parabolic subgroup $Q=P_{\tilde{\omega}_2+\cdots+\tilde{\omega}_{2n-2}}$ which is  the stabilizer of $\Psi(\rho)$ will play an important
role. The Weyl group of the semisimple part of $Q$ is denoted by $\tilde{W}^J$.

\begin{defn}\label{deftau} We define in the Weyl group $\tilde W$ the element $\tau$ as follows:
\begin{equation}\label{tau}
\tau=(s_n s_{n+1} \cdots s_{2n-3} s_{2n-2})(s_{n-1}s_n \cdots s_{2n-4})\cdots
(s_4 s_5 s_6)(s_3 s_4) s_2.
\end{equation}
\end{defn}
It is easy to see that the decomposition is reduced and $\tau$ is a minimal length representative in its class in $\tilde W/\tilde{W}^J$.
Another description of $\tau$ can be given by viewing $\tau$ as a permutation of the set $\{1,\ldots,2n\}$:
\begin{equation}\label{Def:Tau}
\tau (t)=\left\{\begin{array}{cl} n+k&\textrm{ if }t=2k\\k&\textrm{ if }t=2k-1\end{array}\right.\;\;(k=1,2,\cdots,n).
\end{equation}
It follows now immediately from \eqref{Def:Tau}:
\begin{lem}\label{2.1}
In the irreducible $SL_{2n}(\mathbb C)$-representation $\tilde V(\tilde\omega_{2i})=\Lambda^{2i}\mathbb
C^{2n}$ let $v_0$ be the highest weight vector
$v_0=e_1\wedge e_2\wedge\ldots\wedge e_{2i}$. Then (up to sign)
$$
\tau(v_0)=v_\tau=e_1\wedge e_2\wedge\ldots\wedge e_{i}\wedge e_{n+1}\wedge
e_{n+2}\wedge\ldots\wedge e_{n+i}.
$$
\end{lem}
Let $\lambda=b_1\epsilon_1+\ldots +b_{n-1}\epsilon_{n-1}$, $b_1\ge \ldots \ge
b_{n-1}\ge 0$, be a dominant weight for $SL_n(\mathbb C)$.
\begin{lem}\label{2.2}
$\tau(\Psi(\lambda))=b_1\tilde\epsilon_1+\ldots + b_{n-1}\tilde\epsilon_{n-1} +
b_1\tilde\epsilon_{n+1}+\ldots b_{n-1}\tilde\epsilon_{2n-1}.$
\end{lem}
\begin{proof}
Follows directly from Lemma~\ref{2.1} above.
\end{proof}
In Example~\ref{example1} we have introduced a map $\phi:\Phi^+\rightarrow\tilde\Phi^+$
between the positive roots of $\mathfrak{sl}_n$ and the positive roots of $\mathfrak{sl}_{2n}$.
Note that the image of $\alpha=\epsilon_i-\epsilon_j$, $1\le i<j\le n$ is the root
$\phi(\alpha)=\tilde\epsilon_j-\tilde\epsilon_{n+i}$.

\begin{lem}\label{2.4}
\begin{itemize}
\item[{\it i)}] Let $\lambda$ be a dominant weight for $SL_{n}(\mathbb C)$. For a positive
$SL_{2n}$-root $\tilde\alpha$ we have $\langle \tilde \alpha^\vee,\tau(\Psi(\lambda))\rangle<0$
only if the root space of $\tilde\alpha$ lies in $\mathfrak n^{-,a}$.
\item[{\it ii)}] Let $\lambda$ be a dominant $SL_n$-weight and let $\alpha=\epsilon_i-\epsilon_j$ be a positive $SL_n$-root, then:
$$
\langle\alpha^\vee,\lambda \rangle =
-\langle\phi(\alpha)^\vee,\tau(\Psi(\lambda))\rangle.
$$
\item[{\it iii)}] Let $\lambda$ be a dominant weight for $SL_{n}(\mathbb C)$ and let $\tilde\alpha=\tilde\epsilon_p-\tilde\epsilon_q$
be a positive $SL_{2n}$-root. Then $E_{\tilde\alpha} v_{\tau}\not=0$ in $\tilde V(\Psi(\lambda))$ only if
$\tilde\alpha$ is of the form $\tilde\alpha=\tilde\epsilon_{j}-\tilde\epsilon_{n+i}$, $1\le i<j\le n$ and $\langle (\epsilon_i-\epsilon_j)^\vee,\lambda\rangle>0$.
\end{itemize}
\end{lem}
\begin{proof}
Let $\tilde\alpha=\tilde\epsilon_i-\tilde\epsilon_j$ be a positive root. Lemma~\ref{2.2} implies
that for $\lambda=b_1\epsilon_1+\ldots +b_{n-1}\epsilon_{n-1}$ we get
$$
\langle \tilde \alpha^\vee,\tau(\Psi(\lambda))\rangle =
\left\{
\begin{array}{rl}
b_i-b_j \ge 0&\text{if\ }1\le i<j\le n ,\\
b_i-b_{j-n} \ge 0&\text{if\ }1\le i\le n,n+i\le j\le 2n ,\\
b_i-b_{j-n} \le 0&\text{if\ }1\le i\le n,n+1\le j < n+i ,\\
b_{i-n}-b_{j-n} \ge 0&\text{if\ }n+1\le i<j\le 2n, \\
\end{array}
\right.
$$
which proves the lemma.
\end{proof}

The decomposition in \eqref{tau} is reduced, but if we apply $\tau$ to a fundamental weight, then it is
possible to omit some of the reflections. A simple calculation shows:
\begin{lem}\label{Lemma2.3}
Let $\tilde\omega_{2i}$ be the $2i$-th fundamental weight for $SL_{2n}({\mathbb C})$. Then
$$
\tau(\tilde\omega_{2i})=(s_n s_{n+1} \cdots s_{n+i-1})\cdots
(s_{i+2} \cdots  s_{2i+1})(s_{i+1} \cdots s_{2i-1} s_{2i})(\tilde\omega_{2i}).
$$
\end{lem}
Let $L(i)$ be the semisimple part of the Levi subgroup of
$SL_{2n}(\mathbb{C})$ associated with the simple
roots $\tilde\alpha_{i+1},\tilde\alpha_{i+2},\ldots,\tilde\alpha_{i+n-1}$, denote by $\mathfrak{l}(i)$ the Lie algebra
of $L(i)$. Note that $L(i)$ is isomorphic to $SL_n(\mathbb C)$. Let $\varpi_1,\ldots,\varpi_{n-1}$ be the
fundamental weights of $L(i)$,
the enumeration is such that the simple root $\tilde\alpha_{i+j}$ of
$L(i)\subseteq SL_{2n}(\mathbb C)$ corresponds to $\varpi_{j}$.

The restriction of $\tilde\omega_{2i}$ to $L(i)$ is $\varpi_i$. Let $W^{L(i)}$
be the Weyl group of $L(i)$,
we can identify it with the subgroup of the Weyl group of
$SL_{2n}$ generated by the reflections
$s_{i+1},s_{i+2}\ldots, s_{i+n-1}$. Using Lemma~\ref{Lemma2.3}, it is easy to
see:
\begin{lem}\label{Lemma2.4}
$(s_n s_{n+1} \cdots s_{n+i-1})\cdots
(s_{i+2} \cdots  s_{2i+1})(s_{i+1} \cdots s_{2i-1} s_{2i})$ is a reduced
decomposition of the longest word of
$W^{L(i)}$ modulo the stabilizer $W^{L(i)}_{\varpi_i}$ of $\varpi_i$ in
$W^{L(i)}$.
\end{lem}
\section{The fundamental representations - the $\mathfrak{sl}_n$-case}\label{SectionFundSL}
We switch now to Lie algebras and hyperalgebras over $\mathbb Z$. Fix a Chevalley basis for the Lie algebra
$\mathfrak{g}_{\mathbb{Z}}=\mathfrak{sl}_{n,{\mathbb{Z}}} \subset \mathfrak{sl}_{n,\mathbb C}$:
$$
\{f_{\alpha}, e_{\alpha}\ : \ \alpha\in\Phi^+;\ h_1, \ldots, h_{n-1}\},
$$
where $f_{\alpha}\in\mathfrak{g}_{\mathbb{Z},-\alpha}$,
$e_{\alpha}\in\mathfrak{g}_{\mathbb{Z}, \alpha}$ and
$h_i\in\mathfrak{h}_{\mathbb Z}$. For any $m\in\mathbb{Z}_{\geq 1}$, we define
the following elements in $U(\mathfrak{g})$:
\begin{equation}\label{Chevalley}
e_{\alpha}^{(m)}=\frac{e_{\alpha}^m}{m!}\qquad
f_{\alpha}^{(m)}=\frac{f_{\alpha}^m}{m!}\qquad
\binom{h_i}{m}=\frac{h_i(h_i-1)\ldots (h_i-m+1)}{m},
\end{equation}
for $m=0$ we set: $e_{\alpha}^{(0)}=f_{\alpha}^{(0)}=\binom{h_i}{0}=1$.
Recall that the hyperalgebra  $U_{\mathbb{Z}}(\mathfrak{sl}_n)$ of
$(SL_n)_{\mathbb Z}$ is the ${\mathbb Z}$-subalgebra of the complex enveloping algebra
$U(\mathfrak{sl}_n)$ generated by the
elements defined in \eqref{Chevalley}. We will use capital letters to denote the Chevalley basis elements for
$\mathfrak{sl}_{2n,{\mathbb{Z}}}$ (e.g. $E_{\tilde\alpha}$, $F_{\tilde\alpha}$, $H_i$) and
the generators of the hyperalgebra  $U_{\mathbb{Z}}(\mathfrak{sl}_{2n})$  (e.g.
$E_{\tilde \alpha}^{(m)}$, $F_{\tilde\alpha}^{(m)}$, $\binom{H_i}{m}$). Similarly, let $U_{\mathbb{Z}}(\tilde{\mathfrak{b}})$
be the subalgebra generated by all $E_{\tilde\alpha}^{(m)}$ for $m\ge 0$ and $\tilde\alpha>0$,
and all $\binom{H_i}{m}$, $i=1,\ldots,2n-1$, $m\ge 0$.  Denote by $U_{\mathbb{Z}}({\mathfrak{l}}(i))$
the hyperalgebra associated with $\mathfrak{l}(i)$, i.e. it is the subalgebra generated by all
$F_{\tilde\alpha}^{(m)}, E_{\tilde\alpha}^{(m)}$, $m\ge 0,\tilde\alpha>0$
and a root of the Levi subgroup $L(i)$, and by all $\binom{H_j}{m}$, $j=i+1,\ldots,i+n-1$.

Let $\mu$ be a dominant integral weight for $SL_{2n}(\mathbb C)$ and denote by $\tilde{V}(\mu)$
the irreducible $SL_{2n}(\mathbb{C})$-representation of highest weight $\mu$.
Fix a highest weight vector $v_{\mu}$,
the corresponding $\mathbb{Z}$-form is
$\tilde{V}_{\mathbb{Z}}(\mu)=U_{\mathbb{Z}}(\mathfrak{sl}_{2n})v_{\mu}$.
To define the Demazure module $\tilde V_{\mathbb{Z}}(\mu)_w$, fix a representative $\check w$
of $w$ in the simply connected Chevalley group associated with $\mathfrak{sl}_{2n,{\mathbb{Z}}}$
and set $v_{w}:=\check w(v_\mu)$.
The Demazure module $\tilde V_{\mathbb{Z}}(\lambda)_w$
is the cyclic $U_{\mathbb{Z}}(\tilde{\mathfrak{b}})$-subrepresentation
$U_{\mathbb{Z}}(\tilde{\mathfrak{b}}).v_{w}\subseteq  \tilde{V}_{\mathbb Z}(\mu)$.
\begin{lem}\label{coro2.7}
The Demazure module $\tilde V_{\mathbb{Z}}(\ell\tilde\omega_{2i})_\tau$
contained in
$\tilde V_{\mathbb{Z}}(\ell\tilde\omega_{2i})$ is the Weyl module
$V_{\mathbb{Z}}(\ell\varpi_i)$ of highest weight $\ell\varpi_i$ for $U_{\mathbb{Z}}({\mathfrak{l}}(i))$.
\end{lem}
\begin{proof}
Consider $\Psi(\ell\omega_i)=\ell\tilde\omega_{2i}$ and recall that the
restriction of $\tilde\omega_{2i}$ to $\mathfrak{l}(i)$ is $\varpi_i$.
So the $U_{\mathbb{Z}}({\mathfrak{l}}(i))$-submodule $U_{\mathbb{Z}}({\mathfrak{l}}(i))v_{\ell\tilde\omega_{2i}}\subseteq \tilde
V_{\mathbb{Z}}(\ell\tilde\omega_{2i})$
is the Weyl module $V_{\mathbb{Z}}(\ell\varpi_i)$ of highest weight
$\ell\varpi_i$ for $U_{\mathbb{Z}}({\mathfrak{l}}(i))$.
Let $U_{\mathbb{Z}}({\mathfrak{b}}(i))$ be the subalgebra of $U_{\mathbb{Z}}({\mathfrak{l}}(i))$
generated by the  $E_{\tilde\alpha}^{(m)}$, $m\ge 0,\tilde\alpha>0$
and a root of ${\mathfrak{l}}(i)$, and all $\binom{H_j}{m}$, $j=i+1,\ldots,i+n-1$.

The Weyl module $V_{\mathbb{Z}}(\ell\varpi_i)$ is a cyclic
$U_{\mathbb{Z}}({\mathfrak{b}}(i))$-module and is generated by a lowest weight vector
of the form $\check w_{0,i}(v_{\ell\varpi_i})$, where $\check w_{0,i}$ is an appropriate representative (in the Chevalley group
associated with $\mathfrak{l}(i)$) of the longest element $w_{0,i}$ of the Weyl group $W^{L(i)}$ of $\mathfrak{l}(i)$. Recall that $W^{L(i)}$ can be identified with the
subgroup of $\tilde W$ generated by $s_{i+1},\ldots,s_{n+i-1}$.
Now 
$$
\tilde V_{\mathbb Z}(\ell\tilde\omega_{2i})_\tau=
U_{\mathbb{Z}}(\tilde{\mathfrak{b}}) v_{\tau}=
U_{\mathbb{Z}}(\tilde{\mathfrak{b}}) v_{w_{0,i}}=
U_{\mathbb{Z}}({\mathfrak{b}(i)})  v_{w_{0,i}}=
V_{\mathbb{Z}}(\ell\varpi_i)\subseteq  \tilde
V_{\mathbb{Z}}(\ell\tilde\omega_{2i})
$$
by Lemma~\ref{2.4}, Lemma~\ref{Lemma2.3} and Lemma~\ref{Lemma2.4}.
\end{proof}

The previous result implies in particular:
\begin{cor}\label{coro2.5}
$\text{rank\,}\tilde
V_{\mathbb{Z}}(\Psi(\ell\omega_{i}))_\tau=\text{rank\,}V_{\mathbb
Z}(\ell\omega_{i})$.
\end{cor}

Let $\iota: \mathfrak{sl}_n\rightarrow \mathfrak{sl}_n$ be the Chevalley involution
defined by $\iota\vert_\mathfrak{h}=-1$ and $\iota$ exchanges $e_\alpha$ and $-f_\alpha$.
It follows  that $\iota(\mathfrak{n}^-_{\mathbb Z})=\mathfrak{n}^+_{\mathbb Z}$,
and this map extends to an isomorphism of the corresponding hyperalgebras $\iota:U_{\mathbb Z}(\mathfrak{n}^-)
\rightarrow U_{\mathbb Z}(\mathfrak{n}^+)$ and the associated graded versions obtained via the PBW filtration:
$\iota:S^\bullet_{\mathbb Z}(\mathfrak{n}^-)
\rightarrow S^\bullet_{\mathbb Z}(\mathfrak{n}^+)$.

Let $\lambda=\sum a_j \omega_j$ be a dominant weight and set $\lambda^*:=\sum a_j \omega_{n-j}$.
Fix a highest weight vector $v_\lambda\in V_{\mathbb Z}(\lambda)$
and a lowest weight vector $v_{w_0}\in V_{\mathbb Z}(\lambda)$, where $w_0$ is the longest word in the Weyl group
of $\mathfrak{sl}_n$. We get two possible $S^\bullet_{\mathbb Z}(\mathfrak{n}^{-,a})$-structures
on $V_{\mathbb Z}(\lambda)$: one uses the PBW filtration on
$U_{\mathbb Z}(\mathfrak{n}^{-})$ to induce, via the highest weight vector, a PBW filtration
on $V_{\mathbb Z}(\lambda)$ and passes to the associated graded module. One gets the module $V^a_{\mathbb Z}(\lambda)$
discussed before. Now one can do the same also for $U_{\mathbb Z}(\mathfrak{n}^{+})$, once replaced  the  highest weight vector by the lowest weight vector. We denote the cyclic $S^\bullet_{\mathbb Z}(\mathfrak{n}^+)$-module
(generated by the lowest weight vector) by $V^{a,+}_{\mathbb Z}(\lambda)$. Now via $\iota$ this module
also becomes naturally a $S^\bullet_{\mathbb Z}(\mathfrak{n}^-)$-module.

\begin{lem}\label{Chevalleytwist}
The  $S^\bullet_{\mathbb Z}(\mathfrak{n}^-)$-module $V^{a,+}_{\mathbb Z}(\lambda)$ is, as
$S^\bullet_{\mathbb Z}(\mathfrak{n}^-)$-module, isomorphic to $V^{a}_{\mathbb Z}(\lambda^*)$.
\end{lem}
\begin{proof}
Note that twisting the representation map with the Chevalley involution makes the lowest weight vector
(the cyclic generator for the $U(\mathfrak n^+)$-action) into a cyclic generator for the $U(\mathfrak n^-)$-action.
Recall that the Chevalley involution is equal to $-1$ on $\mathfrak h$, so after the twist this is now a highest weight
vector of weight $\lambda^*=-w_0(\lambda)$, where $w_0$ is the longest word in $W$. Since the construction is
compatible with the PBW filtrations with respect to the two algebras, the result for the associated graded modules
follows immediately.
\end{proof}
\begin{prop}\label{proposition2.7}
As $S_{\mathbb Z}(\mathfrak n^{-,a})$-modules, $\tilde
V_{\mathbb{Z}}(\Psi(\ell\omega_{i}))_\tau\simeq
V_{\mathbb{Z}}^a(\ell\omega_{n-i})$.
\end{prop}

\begin{proof}
Let $\mathfrak{n}_i^{-,a}\subseteq \mathfrak{n}^{-,a}$ be the sum of all root subspaces of roots
of the form $\tilde\epsilon_k-\tilde\epsilon_{n+\ell}$, where
$1\le\ell \le i \le k \le n$ and $\ell\not=k$. 
This is a commutative Lie subalgebra, which by Lemma~\ref{2.4} has the following property:
$$
V_{\mathbb{Z}}(\Psi(\ell\omega_{i}))_\tau=U(\tilde{\mathfrak{b}}).v_\tau=U(\mathfrak{n}_i^{-,a}).v_\tau=S^\bullet(\mathfrak{n}_i^{-,a}).v_\tau.
$$
Since $\mathfrak{n}^{-,a}$ is commutative, all root vectors in $\mathfrak{n}^{-,a}$ which are not in $\mathfrak{n}_i^{-,a}$ act trivially on $V_{\mathbb{Z}}(\Psi(\ell\omega_{i}))_\tau$.

Another way to describe $\mathfrak{n}_i^{-,a}$ is as the intersection $\mathfrak{n}^{-,a}\cap \mathfrak{l}(i)$. More precisely, this intersection
is the nilpotent radical of the maximal parabolic subalgebra of $\mathfrak{l}(i)$ associated with the fundamental weight $\varpi_{n-i}$.
By Lemma~\ref{coro2.7} we know that $V_{\mathbb{Z}}(\Psi(\ell\omega_{i}))_\tau= U(\mathfrak{b}(i)).v_\tau\simeq V_{\mathbb{Z}}(\ell\varpi_i)$,
since $v_\tau$ is a lowest weight vector we get in addition
$$
V_{\mathbb{Z}}(\Psi(\ell\omega_{i}))_\tau= U(\mathfrak{n}_i^{-,a}).v_\tau\simeq V_{\mathbb{Z}}(\ell\varpi_i).
$$
Set $\mathfrak n^+(i)= \mathfrak b(i)\cap \tilde{\mathfrak n}^+$, by the isomorphism between $\mathfrak{l}(i)$ and $\mathfrak{sl}_n$ we can
identify $\mathfrak m$ with ${\mathfrak n}^+\subset \mathfrak{sl}_n$.
Consider the associated PBW filtration on $V_{\mathbb{Z}}(\ell\varpi_i)$
by applying the PBW filtration of $U(\mathfrak n^+(i))$ to the lowest weight vector.
Recall that after passing to the associated graded algebra $S^\bullet(\mathfrak n^+(i))$ all root vectors not contained in $\mathfrak{n}_i^{-,a}$
act trivially on $V^{a,+}_{\mathbb{Z}}(\ell\varpi_i)$. Remember that we add a ``$+$" to indicate that this is the associated graded space with respect to
the filtration by the nilpotent radical of the fixed Borel subalgebra and not, as usual, of the opposite nilpotent algebra.
Since $\varpi_i,\varpi_{n-i}$ are cominuscule, $\mathfrak{n}_i^{-,a}$ is commutative
and the PBW filtration on $V_{\mathbb{Z}}(\ell\varpi_i)$ is already a grading. It follows that the action of $\mathfrak{n}_i^{-,a}$ on
$V_{\mathbb{Z}}(\ell\varpi_i)$ and $V^{a,+}_{\mathbb{Z}}(\ell\varpi_i)$ are the same and hence the $\mathfrak{n}_i^{-,a}$
action on $V^{a,+}_{\mathbb{Z}}(\ell\varpi_i)$ and $V_{\mathbb{Z}}(\Psi(\ell\omega_{i}))_\tau$ are isomorphic, and hence so
are the $\mathfrak{n}^{-,a}$ actions by trivial extension. The proposition follows now by Lemma~\ref{Chevalleytwist}.
\end{proof}

\section{The general case for $\mathfrak{sl}_n$}
\subsection{\,} 
We extend Proposition~\ref{proposition2.7} to any dominant weight for $\mathfrak{sl}_n$. 
Recall that for a dominant weight $\lambda=a_1\omega_1+\cdots +a_{n-1}\omega_{n-1}$ we denote by $\lambda^*$ the dominant weight given by $\lambda^*=a_{n-1}\omega_1+\cdots +a_{1}\omega_{n-1}$.
\begin{thm}\label{ThmIsoSLn} Let $\lambda$ be a dominant $\mathfrak{sl}_{n}$-weight. The
Demazure submodule
$\tilde V_{\mathbb Z}(\Psi(\lambda^*))_\tau$ of the
$(\mathfrak{sl}_{2n})_\mathbb{Z}$-module $\tilde V_{\mathbb Z}(\Psi(\lambda^*))$ is,
as an $\mathfrak{n}_{\mathbb{Z}}^{-,a}$-module, isomorphic to $V_{\mathbb Z}^a(\lambda)$.
\end{thm}

The proof of Theorem \ref{ThmIsoSLn} will be given in Section \ref{PROOF} (the strategy of proof is explained in Section \ref{strategy}) . We deduce a useful corollary.

\begin{cor}\label{abelfree}
In particular, $V_{\mathbb Z}^a(\lambda)$ is free as a $\mathbb Z$-module.
\end{cor}
\begin{proof} ({\it of the corollary})
The Demazure module $\tilde V_{\mathbb Z}(\Psi(\lambda^*))_\tau$ is a direct
summand
of the free $\mathbb Z$-module $\tilde V_{\mathbb Z}(\Psi(\lambda^*))$ and hence
free
as $\mathbb Z$-module.
\end{proof}
\subsection{\,}\label{borelaction} The abelianized module $V_{\mathbb{Z}}^a(\lambda)$ is a cyclic module over the
algebra $S_{\mathbb Z}^\bullet(\mathfrak n^{-,a})$
having as a generator the image of a highest weight vector $v_{\lambda}\in V(\lambda)$ in
$V_{\mathbb Z}^a(\lambda)$. 
Hence the module is isomorphic to $S_{\mathbb Z}^\bullet(\mathfrak
n^{-,a})/I_\mathbb{Z}(\lambda)$ where $I_\mathbb{Z}(\lambda)$
is the annihilator of $v_{\lambda}$ in $S_{\mathbb Z}^\bullet(\mathfrak n^{-,a})$.

We have an additional Lie algebra acting on $S_{\mathbb Z}^\bullet(\mathfrak n^{-,a})$
as well as on $V_{\mathbb{Z}}^a(\lambda)$.
Let $\mathfrak b$ be the Borel subalgebra of
$\mathfrak{g}=(\mathfrak{sl}_n)_{\mathbb Z}\otimes\mathbb C$ as in Example
\ref{example1},
so $\mathfrak{g}=\mathfrak n^-\oplus\mathfrak h \oplus\mathfrak n^+$.
As free $\mathbb{Z}$-modules,
$U_{\mathbb{Z}}(\mathfrak n^-)\simeq U_{\mathbb{Z}}(\mathfrak
g)/U^+_{\mathbb{Z}}(\mathfrak h+\mathfrak n^+)$, so that
the adjoint action of $U_{\mathbb{Z}}(\mathfrak b)$ on $U_{\mathbb{Z}}(\mathfrak
g)$ induces a structure of $U_{\mathbb{Z}}(\mathfrak b)$-module on
$U_{\mathbb{Z}}(\mathfrak n^-)$ and hence on
$S_{\mathbb Z}^\bullet(\mathfrak n^{-,a})$. This action is compatible with the induced
$U_{\mathbb{Z}}(\mathfrak b)$-action on $V_{\mathbb Z}^a(\lambda)$ (cf.\cite[Prop.
2.3.]{FFL3}). Recall that for a positive root $\alpha$ we have denoted by
$f_\alpha$ the corresponding fixed Chevalley basis element in
$(\mathfrak{sl}_n)_{-\alpha, \mathbb{Z}}$.
Using the presentation of Demazure modules in terms of generators and relations by Joseph, Mathieu and Polo
(cf. \cite{M}, Lemme 26),
we get as a consequence of the proof of Theorem~\ref{ThmIsoSLn}  the following
description of the ideal  $I_\mathbb{Z}(\lambda)$ (cf. \cite{FFL1}, \cite{FFL3}):
\begin{cor}
The abelianized module $V_{\mathbb Z}^a(\lambda)$ is  isomorphic to
$S_{\mathbb Z}^\bullet(\mathfrak n^{-,a})/I_\mathbb{Z}(\lambda)$ (as a cyclic $S_{\mathbb
Z}(\mathfrak n^{-,a})$-module), where
$$
I_\mathbb{Z}(\lambda)= S_{\mathbb Z}^\bullet(\mathfrak n^{-,a})( U_{\mathbb Z}(\mathfrak
n^+)\circ
\text{span}\{f_\alpha^{(\langle\alpha^\vee,\lambda\rangle+m)}; m\ge 1, \alpha >
0\})\subseteq S_{\mathbb Z}(\mathfrak n^{-,a}).
$$\end{cor}
\subsection{\,}\label{strategy}
The proof of Theorem \ref{ThmIsoSLn} will be given in section~\ref{PROOF}, it needs some preparation.
The strategy of the proof is summarized by the following diagram of $S_{\mathbb Z}^\bullet (\mathfrak{n}^{-,a})$-modules.
For a dominant weight $\lambda=a_1\omega_1+\ldots+a_{n-1}\omega_{n-1}$ (so $\lambda^*=a_{n-1}\omega_1+\ldots+a_{1}\omega_{n-1}$),
we get the following natural maps, the details are described below:
$$
\xymatrix{
S_{\mathbb Z}^\bullet(\mathfrak{n}^{-,a})/I_\mathbb{Z}(\lambda^*)\ar_(.6)\simeq^(.6)h[r]&  V_\ZZ^a(\lambda^*) \ar@{->>}[d]^a \ar[r]^<<<<<<<b &
*!<-5pt,0pt>{\;V_{\mathbb Z}^a(a_{1}\omega^*_{1})\otimes\cdots\otimes V_{\mathbb Z}^a(a_{n-1}\omega_{n-1}^*)} \ar[d]^c_{\wr\mid}\\
S_{\mathbb Z}^\bullet(\mathfrak{n}^{-,a} )/M_\mathbb{Z}(\lambda^*)\ar@{..>>}^f[u]\ar^(.55)g_(.55)\simeq[r]&
\tilde{V}_{\mathbb Z}(\Psi(\lambda))_\tau  \ar@{^{(}->}[r]^<<<<<d&
\tilde{V}_{\mathbb Z}(a_{1}\Psi(\omega_{1}))_\tau\otimes\cdots\otimes  \tilde{V}_{\mathbb Z}(a_{n-1}\Psi(\omega_{n-1}))_\tau.
}
$$
Let us describe the diagram above and the strategy of the proof. We recall that given a tensor product of cyclic $S^\bullet_\ZZ (\mathfrak{n}^{-,a})$-modules, 
the Cartan component of the tensor product is, by definition, the cyclic $S^\bullet_\ZZ (\mathfrak{n}^{-,a})$-submodule
generated by the tensor product of the cyclic generators. Further, recall that 
the isomorphism $V_{\mathbb Z}^a(\ell\omega^*_{j})\simeq  \tilde{V}_{\mathbb Z}(\ell\Psi(\omega_{j}))_\tau$
sends the highest weight vector $v_{\ell\omega^*_j}$ to the extremal weight vector $v_{\tau(\ell\Psi(\omega_{j}))}$ and uses the Chevalley involution.
The maps above are defined as follows:
\begin{itemize}
\item $b$ is induced by the compatibility of the PBW filtration with the tensor product, and it is surjective onto the Cartan component of this tensor product.
\item $I_\mathbb{Z}(\lambda^*)$ is the annihilator in $S_{\mathbb Z}^\bullet(\mathfrak{n}^{-,a})$ of the image of the highest weight vector
$v_{\lambda^*}$ in $V^a_{\mathbb{Z}}(\lambda^*)$ and $h$ is the corresponding quotient map.
\item $c$ is the isomorphism given by Proposition~\ref{proposition2.7}.
\item $d$ is the isomorphism onto the Cartan component of the tensor product. The fact that this is an isomorphism follows
by standard monomial theory \cite{LMS} or Frobenius splitting \cite{R}.
\item $a$ equals $c\circ b$ after identifying $\tilde{V}_{\mathbb Z}(\Psi(\lambda))_\tau$ with its image under $d$.
\item $M_\mathbb{Z}(\lambda^*)$  is the annihilator in $S_{\mathbb Z}^\bullet(\mathfrak{n}^{-,a})$ of the  extremal weight vector
$v_{\tau(\Psi(\lambda))}$ in $\tilde{V}_{\mathbb Z}(\Psi(\lambda))_\tau$ and $g$ is the corresponding quotient map.
\item $f$ is going to be constructed in the proof.
\end{itemize}
In order to finish the proof we will show that $M_\mathbb{Z}(\lambda^*)\subseteq I_\mathbb{Z}(\lambda^*)$, and the inclusion induces the surjective map $f$ which in turn shows that the map $a$ is an isomorphism.

\subsection{\,}\label{generatordemzureideal}
 We first determine $M_\mathbb{Z}(\lambda^*)$. By \cite[Lemme 26]{M}, the Demazure
module  $\tilde V_{\mathbb Z}(\Psi(\lambda))_\tau$ is isomorphic
to the algebra $U_{\mathbb{Z}}(\tilde{\mathfrak n}^+)$ modulo the left ideal
$\tilde I_{\mathbb{Z}}(\tau\Psi(\lambda))$ generated by the
$E_{\tilde\alpha}^{(m)}$  for all $m\ge 1$ if $\langle
{\tilde\alpha}^\vee,\tau\Psi(\lambda)\rangle\ge 0$,
and $E_{\tilde\alpha}^{(-\langle {\tilde\alpha}^\vee,\tau\Psi(\lambda)\rangle +
m)}$ for all $m\ge 1$
otherwise.

\subsection{\,} 
The annihilator $M_\mathbb{Z}(\lambda^*)$ is the intersection
of $U_{\mathbb Z}(\mathfrak n^{-,a})\subset U_{\mathbb Z}(\tilde{\mathfrak n}^+)$
with the ideal $\tilde I_{\mathbb{Z}}(\tau\Psi(\lambda))$.
To determine the intersection, let us divide the positive roots of
$SL_{2n}$ into three families:
\begin{itemize}
\item[-] $\tilde \alpha$ is of the \emph{first type} if $\tilde \alpha=\phi(\alpha)$ for some positive $SL_n$-root $\alpha$.
\item[-]  $\tilde \alpha=\tilde \epsilon_k-\tilde \epsilon_l$  is of
\emph{second type} if  $1\le k<l\le n$ or
$n+1\le k<l\le 2n$.
\item[-] $\tilde \alpha=\tilde \epsilon_k-\tilde \epsilon_l$   is of \emph{third
type} if $1\le k\le n$, $n+1\le l\le 2n$ and $k<l-n$.
\end{itemize}
$$
\left(\begin{array}{ccccccccc}
\ddots &   & 2nd  &   &  | & \ddots  &   & 3rd  &   \\
  & \ddots  &  type &   &  | &   & \ddots  & type  &   \\
  &   &  \ddots &   &  | & 1st  &   & \ddots  &   \\
  &   &   &  \ddots &  | & type  &   &   & \ddots  \\
 - & -  & -& -  &  | & -  & -  & -  & -  \\
  &   &   &   &  | & \ddots  &   & 2nd  &   \\
  &   &   &   &  | &   & \ddots  & type  &   \\
  &   &   &   &  | &   &   & \ddots  &   \\
  &   &   &   &  | &   &   &   &\ddots  \end{array}\right)
$$
The $E_{\tilde\alpha}$, ${\tilde\alpha}$ of second type, span a Lie subalgebra
isomorphic to two copies of $\mathfrak b_{\mathbb Z}$. Let us denote the first copy
spanned by the $E_{\tilde\alpha}$, $\tilde \alpha=\tilde \epsilon_k-\tilde \epsilon_l$, $1\le k<l\le n$ 
by $\mathfrak b^1_{\mathbb Z}$ and denote the second copy
spanned by the $E_{\tilde\alpha}$, $\tilde \alpha=\tilde \epsilon_k-\tilde \epsilon_l$, $n+1\le k<l\le 2n$,
by $\mathfrak b^2_{\mathbb Z}$. 

Let $\tilde I_{\mathbb{Z}}(\infty)\subset U_{\mathbb Z}(\tilde{\mathfrak n}^+)$ be the left 
$U_{\mathbb Z}(\tilde{\mathfrak n}^+) $-submodule generated by the $E^{(m)}_{\tilde\alpha}$, $m\ge 1$, 
${\tilde\alpha}$ of second or third type, then Lemma~\ref{2.4} and a PBW basis argument show that we have the following $\mathbb Z$-module
decomposition: 
$$
U_{\mathbb Z}(\tilde{\mathfrak n}^+)= U_{\mathbb Z}(\mathfrak n^{-,a})\oplus \tilde I_{\mathbb{Z}}(\infty)=
S^\bullet_{\mathbb Z}(\mathfrak n^{-,a})\oplus \tilde I_{\mathbb{Z}}(\infty)
\quad\text{and}\quad
\tilde I_{\mathbb{Z}}(\infty)\subset \tilde I_{\mathbb{Z}}(\tau\Psi(\lambda)).
$$
By abuse of notation we identify in the following $S^\bullet_{\mathbb Z}(\mathfrak n^{-,a})$ with $U_{\mathbb Z}(\tilde{\mathfrak n}^+)/\tilde I_{\mathbb{Z}}(\infty)$.
So determining $M_\mathbb{Z}(\lambda^*)=U_{\mathbb Z}(\mathfrak n^{-,a})\cap \tilde I_{\mathbb{Z}}(\tau\Psi(\lambda))$ 
(the intersection taking place in $U_{\mathbb Z}(\tilde{\mathfrak n}^+)$)
is equivalent to determining the image of $ \tilde I_{\mathbb{Z}}(\tau\Psi(\lambda))/ I_{\mathbb{Z}}(\infty)$ in $S^\bullet_{\mathbb Z}(\mathfrak n^{-,a})$. 
In the following we identify $M_\mathbb{Z}(\lambda^*)$ with
$ \tilde I_{\mathbb{Z}}(\tau\Psi(\lambda))/ I_{\mathbb{Z}}(\infty)$.

Note that $U_{\mathbb Z}(\mathfrak b^1_{\mathbb Z}\oplus \mathfrak b^2_{\mathbb Z})$ acts naturally via the adjoint action
on $\tilde{\mathfrak n}^+_{\mathbb Z}$ and hence on $U_{\mathbb Z}(\tilde{\mathfrak n}^+)$. The span of the $E_{\tilde \alpha}$,
$\tilde\alpha$ of second or third type is stable under this adjoint action of $\mathfrak b^1_{\mathbb Z}\oplus \mathfrak b^2_{\mathbb Z}$,
so $\tilde I_{\mathbb{Z}}(\infty)\subset U_{\mathbb Z}(\tilde{\mathfrak n}^+)_{\mathbb Z}$ is a submodule with respect to this adjoint action.

We get an induced $U_{\mathbb Z}(\mathfrak b^1_{\mathbb Z}\oplus \mathfrak b^2_{\mathbb Z})$-action on $S^\bullet_{\mathbb Z}(\mathfrak n^{-,a})$
which we denote by ``$\circ$''. Moreover, since $U^+_{\mathbb Z}(\mathfrak b^1_{\mathbb Z}\oplus \mathfrak b^2_{\mathbb Z})$
(the elements without constant term) is contained in $ \tilde  I_{\mathbb{Z}}(\infty)$, we see that 
 $M_\mathbb{Z}(\lambda^*)=\tilde I_{\mathbb{Z}}(\tau\Psi(\lambda))/\tilde I_{\mathbb{Z}}(\infty)$ is a 
 $U^+_{\mathbb Z}(\mathfrak b^1_{\mathbb Z}\oplus \mathfrak b^2_{\mathbb Z})$-stable submodule with respect to the ``$\circ$''-action of
$U_{\mathbb Z}(\mathfrak b^1_{\mathbb Z}\oplus \mathfrak b^2_{\mathbb Z})$.
As a first step in the proof of the theorem we show:
\begin{lem}
The left $S^\bullet_{\mathbb Z}(\mathfrak n^{-,a})$-submodule $M_\mathbb{Z}(\lambda^*)\subset S^\bullet_{\mathbb Z}(\mathfrak n^{-,a})$
is generated by 
$$
m_\mathbb{Z}(\lambda^*):=\langle U_{\mathbb Z}(\mathfrak b^1_{\mathbb Z}\oplus \mathfrak b^2_{\mathbb Z})\circ 
E_{\tilde\alpha}^{(-\langle {\tilde\alpha}^\vee,\tau\Psi(\lambda)\rangle +
m)}\mid \tilde\alpha\text{\ of first type and $m\ge 1$}\rangle_{\mathbb Z}.
$$
\end{lem}
\begin{proof}
Let $\bar{\tt{m}}$ be an element of $M_\mathbb{Z}(\lambda^*)$
and choose a representative ${\tt{m}}$ in $\tilde I_{\mathbb{Z}}(\tau\Psi(\lambda))$. Since we are free to choose
a representative modulo $I_{\mathbb{Z}}(\infty)$, we may assume (see section~\ref{generatordemzureideal}) that $\tt m$ is a sum of monomials
of the form $r E_{\tilde\alpha}^{(\ell)}$ where $\ell= -\langle {\tilde\alpha}^\vee,\tau\Psi(\lambda)\rangle +
k$ for some $k\ge 1$ and $\tilde\alpha$ of first type, and $r$ is a monomial in the
$E_{\tilde\beta}^{(q)}$, $q\ge 0$ and $\tilde\beta$ of first, second or third type.

If $\tilde\gamma$ is a root of third type and ${\tilde\beta}$ is any other positive root, then
$[E_{\tilde\gamma},E_{\tilde\beta}]=c E_{\tilde\gamma'}$, where $c\in \mathbb Z$ and either $c=0$ or $\tilde\gamma'$ of
third type. So if $r$ has a factor $E_{\tilde\gamma}^{(p)}$,  $p>0$ and $\tilde\gamma$
a root of  third type, then we can rewrite the
monomial $r E_{\tilde\alpha}^{(\ell)}$ as a sum of monomials of the form $r'E_{\tilde\gamma}^{(p')}$, $p'>0$.
Since this sum is an element in  $I_{\mathbb{Z}}(\infty)$, without loss of generality we will assume in the following that 
$r$ has only factors of the form $E_{\tilde\beta}^{(\ell)}$, $\tilde\beta$ of first or second type.

If $\tilde\gamma$ is of second type and $\tilde\beta$ is of first type, then 
$[E_{\tilde\gamma},E_{\tilde\beta}]=c E_{\tilde\gamma'}$, where $c\in \mathbb Z$ and either $c=0$ or  
$\tilde\gamma'$ is of first or third type. So after reordering the factors
we can assume without loss of generality in the following that $r E_{\tilde\alpha}^{(\ell)}$ is of the form
$r=r_1r_{2}$, where $r_1$ is a monomial in the $E_{\tilde\beta}^{(\ell)}$, $\tilde\beta$ of first type,
and $r_2$ is a monomial in the $E_{\tilde\gamma}^{(\ell)}$, $\tilde\gamma$ of second type.

Recall that for $\tilde\gamma$ of second type we have
$$
\begin{array}{rcl}
E_{\tilde\gamma}E_{\tilde\beta_1}\cdots E_{\tilde\beta_m}&\equiv&
\sum_{i=1}^{m} E_{\tilde\beta_1}\cdots E_{\tilde\beta_{i-1}}(E_{\tilde\gamma}E_{\tilde\beta_i}-E_{\tilde\beta_i} E_{\tilde\gamma})E_{\tilde\beta_{i+1}}\cdots E_{\tilde\beta_m} 
\bmod I_{\mathbb{Z}}(\infty)\\
&\equiv&\sum_{i=1}^{m} E_{\tilde\beta_1}\cdots E_{\tilde\beta_{i-1}}\big(ad(E_{\tilde\gamma})(E_{\tilde\beta_i})\big)E_{\tilde\beta_{i+1}}\cdots E_{\tilde\beta_m}
\bmod I_{\mathbb{Z}}(\infty)\\
&\equiv&ad(E_{\tilde\gamma})\big(E_{\tilde\beta_1}\cdots E_{\tilde\beta_m}\big)
\bmod I_{\mathbb{Z}}(\infty).
\end{array}
$$
An appropriate reformulation of the equality above holds also for the divided powers of the root vectors. It follows that 
$r_2 E_{\tilde\alpha}^{(\ell)}\in m_\mathbb{Z}(\lambda^*)$ and hence 
$r E_{\tilde\alpha}^{(\ell)}=r_1 r_2 E_{\tilde\alpha}^{(\ell)}\in S^\bullet_{\mathbb Z}(\mathfrak n^{-,a})m_\mathbb{Z}(\lambda^*)$,
which implies that $M_\mathbb{Z}(\lambda^*)$ is generated by $m_\mathbb{Z}(\lambda^*)$
as a left $S^\bullet_{\mathbb Z}(\mathfrak n^{-,a})$-module.
\end{proof}
\subsection{\,}
To compare $M_\mathbb{Z}(\lambda^*)$ with $I_\mathbb{Z}(\lambda^*)$, we need a variant of the description
of $m_\mathbb{Z}(\lambda^*)$. Let $\Delta (\mathfrak b_\mathbb{Z})\subset \mathfrak b^1_{\mathbb Z}\oplus \mathfrak b^2_{\mathbb Z}$
be the Lie subalgebra obtained as a diagonally embedded copy of $\mathfrak b_\mathbb{Z}$. Let
$U_{\mathbb Z}(\Delta (\mathfrak b_\mathbb{Z}))\subset U_{\mathbb Z}( (\mathfrak b_\mathbb{Z}^1)\oplus  (\mathfrak b_\mathbb{Z}^2))$
be its hyperalgebra.
\begin{lem}\label{deltabandb}
$m_\mathbb{Z}(\lambda^*)=
\langle U_{\mathbb Z}(\Delta (\mathfrak b_\mathbb{Z})) \circ E_{\tilde\alpha}^{(-\langle\tilde\alpha^\vee, \tau\Psi(\lambda)\rangle +m)}\mid \tilde\alpha
\text{ of first type}, m\ge 1\rangle_\mathbb{Z}$.
\end{lem}
\begin{proof}
We assume first that $k$ is an algebraically closed field of arbitrary characteristic.
Let $B$ be the subgroup of upper triangular invertible matrices in $SL_n(k)$, so $\text{Lie\,}B=\mathfrak b$.
Let $B^1\times B^2\subset SL_{2n}(k)$ be the subgroup with Lie algebra $\mathfrak b^1\oplus \mathfrak b^2$
and denote by $\Delta(B)\subset B^1\times B^2$ the diagonally embedded group isomorphic to $B$.

Let $\mathfrak q$ be the sum of the $SL_{2n}$-root spaces corresponding to roots of second or third  type. 
Then $\tilde{\mathfrak n}^+=\mathfrak n^{-,a}\oplus \mathfrak q$
and we identify $\mathfrak n^{-,a}$ with $\tilde{\mathfrak n}^+ / \mathfrak q$. The adjoint action of $B^1\times B^2$ on
$\mathfrak{sl}_{2n}$ admits $\tilde{\mathfrak n}^+$ as well as $\mathfrak q$ as submodules, so we get an induced 
$B^1\times B^2$-action on $\mathfrak n^{-,a}=\tilde{\mathfrak n}^+ / \mathfrak q$. This action naturally extends
to the commutative hyperalgebra $S_k^\bullet(\mathfrak n^{-,a})$.

If we replace the group action of $B^1\times B^2$ by the induced action of the hyperalgebra 
$U_k(\mathfrak b^1\oplus \mathfrak b^2)$ of the group, then we get the action of $U_k(\mathfrak b^1\oplus \mathfrak b^2)$
on $U_k(\tilde{\mathfrak n}^+)$ respectively on $S^\bullet_k(\mathfrak n^{-,a})$ discussed above, similarly for the action of $\Delta(B)$
and its hyperalgebra  $U_{k}(\Delta (\mathfrak b))$. Recall that for a root $\tilde\alpha$ of type 1 we have
$$
U_k(\mathfrak b^1\oplus \mathfrak b^2)\circ E_{\tilde\alpha}^{(m)} =\langle \text{Ad}((b^1,b^2))\circ (E_{\tilde\alpha})^{(m)}\mid (b_1,b_2)\in B^1\times B^2\rangle,
$$
i.e. the smallest $U_k(\mathfrak b^1\oplus \mathfrak b^2)$ stable subspace containing $(E_{\tilde\alpha})^{(m)}$ is the linear span of the $B^1\times B^2$-orbit. The same
holds in the other case, so we have:
$$
U_k(\Delta(\mathfrak b))\circ E_{\tilde\alpha}^{(m)} =\langle \text{Ad}((b,b))\circ (E_{\tilde\alpha})^{(m)}\mid b\in B\rangle.
$$
Let $\mathfrak d$ be the sum of the $SL_{2n}$-root spaces corresponding to roots of first or third  type
and let $\mathfrak d^3$ be just the sum of the root spaces corresponding to roots  of third  type, so $\mathfrak d=\mathfrak n^{-,a}\oplus \mathfrak d^3$.
We identify $\mathfrak d\subset \mathfrak{sl}_{2n} $ with $M_n(k)$, formally this can be done by the map
$$
\chi:\mathfrak d\rightarrow M_n(k), \quad  \tilde A=\left(\begin{array}{cc}0 & A \\ 0 & 0\end{array}\right)\mapsto A,
$$
where $A$ is a $n\times n$ matrix.
In the following we simplify the notation and omit the map $\chi$. We freely identify $\mathfrak d$ with $M_n(k)$, so we denote by $A$
the $n\times n$ matrix as well as the $2n\times 2n$-matrix $\tilde A\in \mathfrak d$. Note that
for $(b_1,b_2)\in B^1\times B^2$ we get
$$
\text{Ad}((b_1,b_2))\circ \tilde A=\left(\begin{array}{cc}b_1 & 0 \\ 0 & b_2\end{array}\right)\left(\begin{array}{cc}0 & A \\ 0 & 0\end{array}\right)  
\left(\begin{array}{cc}b^{-1}_1 & 0 \\ 0 & b^{-1}_2\end{array}\right)=
\left(\begin{array}{cc}0 & b_1  {A} b_2^{-1}  \\ 0 & 0\end{array}\right),
$$
we just write $\text{Ad}((b_1,b_2))\circ (A)=b_1  {A} b_2^{-1}$ and  $\text{Ad}((b,b))\circ (A)=b   A b^{-1}$.
Recall that $\chi$ is just a vector space isomorphism. If we equip in addition $M_n(k)$ with the trivial Lie bracket, then this becomes
also a Lie algebra homomorphism. In this sense we identify also the (commutative) Lie subalgebras $\mathfrak{n}^{-,a}$ and $\mathfrak d^3$
with subalgebras of $M_n$. 
An elementary calculation shows how the $B^1\times B^2$-orbit through $E_{\tilde\alpha} = E_{\epsilon_{n}-\epsilon_{n+1}}$ 
breaks up into $\Delta(B)$-orbits, recall that we identify $\mathfrak d$ with $M_n(k)$ and  $E_{\epsilon_{n}-\epsilon_{n+1}}\in \mathfrak d$ corresponds to $E_{n,1}$:
$$
\begin{array}{rcl}
\{\text{Ad}((b_1,b_2))\circ (E_{\tilde\alpha})\mid (b_1,b_2)\in B^1\times B^2\}
&=&\{ b_1 E_{n,1} b^{-1}_2\mid b_1,b_2\in B\} \\
&=&\bigcup_{\lambda\in k} 
\{ b (E_{n,1}+\lambda E_{1,1}) b^{-1}\mid b \in B\} \\&\subset& M_n(k).
\end{array}
$$
We conclude for the linear span:
$$
\begin{array}{rcl}
U_k(\mathfrak b^1\oplus \mathfrak b^2)\circ E_{\tilde\alpha}^{(m)} &=&\langle \text{Ad}((b_1,b_2))\circ\big(E_{\tilde\alpha})^{(m)}\mid (b_1,b_2)\in B^1\times B^2\rangle \\
&=&\langle \big(\text{Ad}((b_1,b_2))\circ(E_{\tilde\alpha})\big)^{(m)}\mid (b_1,b_2)\in B^1\times B^2\rangle \\
&=&  \langle \big( b_1 E_{n,1} b^{-1}_2\big)^{(m)}\mid b_1,b_2\in B\rangle \\
&=&  \langle \bigcup_{\lambda\in k}  \{ b (  E_{n,1}+\lambda   E_{1,1})^{(m)} b^{-1}\mid b \in B\}\rangle. \\
\end{array}
$$
Let $I(\mathfrak d^3)\subset U_k(\mathfrak d)$ be the left ideal in the hyperalgebra generated by $\mathfrak d^3$.
It follows:
$$
\begin{array}{rcl}
U_k(\mathfrak b^1\oplus \mathfrak b^2)\circ E_{\tilde\alpha}^{(m)} \bmod I(\mathfrak d^3)&=& 
\langle \bigcup_{\lambda\in k}  \{ \big(b (  E_{n,1}+\lambda   E_{1,1}) b^{-1}\big)^{(m)}\mid b \in B\}\rangle \bmod I(\mathfrak d^3) \\
&=& \langle  \big( b   E_{n,1}b^{-1}\big)^{(m)} \mid b\in B\rangle \bmod I(\mathfrak d^3)\\
\end{array}
$$
because the $E^{(\ell)}_{1,1}$, $\ell\ge 1$, lie in the $\Delta(B)$-stable ideal $I(\mathfrak d^3)$. It follows:
$$
\begin{array}{rcl}
U_k(\mathfrak b^1\oplus \mathfrak b^2)\circ E_{\tilde\alpha}^{(m)} \bmod I(\mathfrak d^3)
&=& \langle \text{Ad}((b,b))(E_{\tilde\alpha})^{(m)}\mid b\in B\rangle \bmod I(\mathfrak d^3)\\
&=& U_k(\Delta(\mathfrak b))\circ E_{\tilde\alpha}^{(m)}\bmod I(\mathfrak d^3).
\end{array}
$$
Since $I(\mathfrak d^3)\subset \tilde I_k(\infty)$, the equation holds in 
$S^\bullet(\mathfrak{n}^{-,a})=U_k(\tilde{\mathfrak{n}}^+) / \tilde I(\infty)$ too and hence
$$
U_k(\mathfrak b^1\oplus \mathfrak b^2)\circ E_{\tilde\alpha}^{(m)} =
U_k(\Delta(\mathfrak b))\circ E_{\tilde\alpha}^{(m)}
$$
in $S^\bullet(\mathfrak{n}^{-,a})$.
It is now easy to see that the same arguments prove the equality for all $E_{\tilde\alpha}^{(m)}$, $\tilde\alpha$
of first type and $m\ge 1$. We have clearly
$$
U_{\mathbb Z}(\mathfrak b^1\oplus \mathfrak b^2)\circ E_{\tilde\alpha}^{(m)} \supseteq
U_{\mathbb Z}(\Delta(\mathfrak b))\circ E_{\tilde\alpha}^{(m)}.
$$
Since we have equality after base change for fields of arbitrary characteristics, the equality of the modules holds also over $\mathbb Z$. 
In particular, the following equality holds  in $S^\bullet_{\mathbb Z}(\mathfrak{n}^{-,a})$:
$$
\begin{array}{rcl}
m_\mathbb{Z}(\lambda^*)&=&\langle U_{\mathbb Z}(\mathfrak b^1_{\mathbb Z}\oplus \mathfrak b^2_{\mathbb Z})\circ 
E_{\tilde\alpha}^{(-\langle {\tilde\alpha}^\vee,\tau\Psi(\lambda)\rangle +
m)}\mid \tilde\alpha\text{\ of first type and $m\ge 1$}\rangle_{\mathbb Z}\\
&=& \langle U_{\mathbb Z}(\Delta (\mathfrak b_\mathbb{Z})) \circ E_{\tilde\alpha}^{(-\langle\tilde\alpha^\vee, \tau\Psi(\lambda)\rangle +m)}\mid \tilde\alpha
\text{ of first type and } m\ge 1\rangle_\mathbb{Z}.
\end{array}
$$
\end{proof}
\subsection{Proof of Theorem~\ref{ThmIsoSLn}}\label{PROOF}
Recall the identification of the abelianized version of $\mathfrak n^-\subset \mathfrak{sl}_n$ with $\mathfrak{n}^{-,a}\subset \mathfrak{sl}_{2n}$,
which sends the image of a Chevalley generator $f_\alpha$ to $E_{\phi(\alpha)}$. By  Lemma~\ref{2.4} (see also Lemma~\ref{Chevalleytwist}
for the twist $\lambda\leftrightarrow\lambda^*$) the elements $E_{\tilde\alpha}^{(-\langle\tilde\alpha^\vee, \tau\Psi(\lambda)\rangle +m)}$, 
where $\tilde\alpha$ is of first type and $m\ge 1$ are elements of $I_\mathbb{Z}(\lambda^*)$. Now the $U_{\mathbb{Z}}(\mathfrak b)$-module 
structure on $S_{\mathbb Z}^\bullet(\mathfrak n^{-,a})$ described in section~\ref{borelaction} is the same as the one described above,
so it follows that $m_\mathbb{Z}(\lambda^*)\subset I_\mathbb{Z}(\lambda^*)$ and hence $M_\mathbb{Z}(\lambda^*)\subset I_\mathbb{Z}(\lambda^*)$,
which, as explained in Section \ref{strategy}, finishes the proof of the theorem.
\qed

\subsection{\,} Let $\rho$ be the sum of all fundamental weights for $SL_n$ and set
$\tilde\rho=\Psi(\rho)$.
Let $Q_{\mathbb Z}\subset (SL_{2n})_{\mathbb Z}$ be the corresponding parabolic
$\mathbb Z$-subgroup. Recall that  $(N^{-,a})_\mathbb{Z}$ is a commutative
subgroup of the Borel subgroup $\tilde B_{\mathbb Z}$. For any $SL_{2n}$-root
$\tilde \alpha$ let
$U_{\mathbb Z, \tilde \alpha}$ be the associated root subgroup.
\begin{lem}\label{N^{-,a}-orbit}
The orbit $\tilde B_{\mathbb Z}.\tau\subset
(SL_{2n})_{\mathbb Z}/Q_{\mathbb Z}$ is equal to $N^{-,a}.\tau$,
and the map $N^{-,a}\rightarrow N^{-,a}.\tau$, $u\mapsto u\tau$, is a bijection.
\end{lem}
\begin{proof}
We have $\tilde B_{\mathbb Z}.\tau=\prod_{\tilde\alpha>0} U_{\mathbb
Z,\tilde\alpha}.\tau$, and the map
$\prod_{\tilde\alpha\in \Gamma} U_{\mathbb Z, \tilde\alpha}\rightarrow
\prod_{\tilde\alpha\in \Gamma} U_{\mathbb Z,\tilde\alpha}.\tau$ is a bijection,
where $\Gamma$ is the set of all positive roots of $SL_{2n}$ such that
$\tau^{-1}(\tilde\alpha)<0$ and $\tau^{-1}(\tilde\alpha)$ is not
an element of the root system of $Q_{\mathbb Z}$. Now this condition is
fulfilled if and only if
$\langle\tau^{-1}(\tilde\alpha^\vee),\tilde\rho\rangle<0$, or, equivalently,
$\langle\tilde\alpha^\vee,\tau(\tilde\rho)\rangle<0$. By Lemma~\ref{2.2} this is
only possible if $\tilde\alpha=\tilde \epsilon_i-\tilde\epsilon_j$, $i<j$, is
such that $1\le i\le n$,
$n+1\le j\le 2n$ and $i>j-n$. But this implies that the root subgroup
$U_{\mathbb Z,\tilde\alpha}$ is a subgroup of $N^{-,a}$,
and all root subgroups of $(SL_{2n})_{\mathbb Z}$ contained in $N^{-,a}$ satisfy
this condition. It follows that $N^{-,a}$ is the product of all root subgroups
corresponding to positive
roots of $SL_{2n}$ in $\Gamma$.
\end{proof}

Recall that the degenerate flag scheme ${\mathcal{F}\ell}(\lambda)_{\mathbb
Z}^a$ is the closure of the
$N_{\mathbb Z}^{-,a}$-orbit through the highest weight vector in $\mathbb
P(V_{\mathbb Z}^a(\lambda))$.
\begin{thm}
Let $\lambda$ be a dominant weight for $SL_n$.  The Schubert scheme
$X_{\mathbb Z}(\tau)\subset$ $\mathbb P(\tilde V_{\mathbb Z}(\Psi(\lambda))_\tau)$
is
isomorphic to the degenerate partial flag scheme
${\mathcal{F}\ell}(\lambda^*)_{\mathbb Z}^a$ for $(SL_n)_{\mathbb Z}$, and this
isomorphism induces a
module isomorphism $H^0(X_{\mathbb Z}(\tau),\mathcal L_{\Psi(\lambda)})\simeq
(V_{\mathbb Z}^a(\lambda^*))^*$.
\end{thm}
\begin{proof} We consider only the case where $\lambda$ is regular, the arguments in the general case are similar.
With respect to the isomorphism in Lemma~\ref{N^{-,a}-orbit}, the orbit
$$\tilde B_{\mathbb Z}.\tau\subset (SL_{2n})_{\mathbb Z}/(\tilde
P_\lambda)_{\mathbb Z}\hookrightarrow \mathbb P(\tilde V(\Psi(\lambda))),$$
through the extremal weight vector, which is the same as the $N^{-,a}$-orbit, is
mapped onto the $N_{\mathbb Z}^{-,a}$-orbit through
the highest weight vector in $\mathbb P(V_{\mathbb Z}^a(\lambda^*))$. By
definition, the Schubert scheme $X_{\mathbb Z}(\tau)$ is the closure of
the orbit $\tilde B_{\mathbb Z}.\tau$ and the degenerate
flag scheme ${\mathcal{F}\ell}(\lambda^*)_{\mathbb Z}^a$ is the closure of the
$N_{\mathbb Z}^{-,a}$-orbit.
It follows that the module isomorphism induces an isomorphism
between the Schubert scheme $X_{\mathbb Z}(\tau)\subseteq \mathbb P(\tilde
V_{\mathbb Z}(\Psi(\lambda))_\tau)$ and the degenerate
flag scheme ${\mathcal{F}\ell}(\lambda^*)_{\mathbb Z}^a$ in $\mathbb P(V_{\mathbb
Z}^a(\lambda))$. Hence we get induced isomorphisms
$$
H^0(X_{\mathbb Z}(\tau),\mathcal L_{\Psi(\lambda)})\simeq (\tilde V_{\mathbb
Z}(\Psi(\lambda))_\tau)^*\simeq (V_{\mathbb Z}^a(\lambda^*))^*
$$
for the dual modules.
\end{proof}
Let $k$ be an algebraically closed field of arbitrary characteristic and denote
by
$V_k(\lambda)=V_{\mathbb Z}(\lambda)\otimes_{\mathbb Z} k$,
$U_k({\mathfrak{sl}}_n)=U_{\mathbb Z}({\mathfrak{sl}}_n)\otimes_{\mathbb Z} k$,
$U_k({\mathfrak{n}}^-)=U_{\mathbb Z}(\mathfrak{n}^-)\otimes_{\mathbb Z} k$ etc.
the objects obtained by base change.
The PBW filtration
$$
V_k(\lambda)_\ell=\langle Y_1^{(m_1)}\cdots Y_N^{(m_N)}v_\lambda \mid
m_1+\ldots+m_N\le \ell, Y_1,\ldots,Y_N\in {\mathfrak{n}}^-_k\rangle
$$
and the associated graded space $V^a_k(\lambda)$ is defined in the same way as
before, and by Corollary~\ref{abelfree} we have
$V_k(\lambda)_\ell=V_{\mathbb Z}(\lambda)_\ell\otimes_{\mathbb Z} k$ and
$V^a_k(\lambda)=V^a_{\mathbb Z}(\lambda)\otimes_{\mathbb Z} k$.
The group $N_{k}^{-,a}$ acts on the abelianized representation $V_k^a(\lambda)$,
and the degenerate flag variety ${\mathcal{F}\ell}(\lambda)_{k}^a$
is the closure of the $N_{k}^{-,a}$-orbit through the highest weight vector in
$\mathbb P(V_{k}^a(\lambda))$.

Now by the results of \cite{M,MR,R,RR} one knows that for Demazure modules we
have $\tilde V_k(\Psi(\lambda))_\tau=
\tilde V_{\mathbb Z}(\Psi(\lambda))_\tau\otimes_{\mathbb Z} k$,
$X_k(\tau)=X_{\mathbb Z}(\tau)\otimes_{\mathbb Z} k$, etc., and 
that the Schubert varieties are Frobenius split, projectively normal and have
rational singularities. It follows that
$V^a_k(\lambda^*)=\tilde V_k(\Psi(\lambda))_\tau$ and
${\mathcal{F}\ell}(\lambda^*)_{k}^a =X_k(\tau)$, so
the degenerate flag variety has in this case the same nice geometric properties
as the Schubert variety.
For a dominant $SL_n$-weight $\lambda=\sum_{i=1}^{n-1} a_i\omega_i$ let the
support $\text{supp\,} \lambda$ of $\lambda$  be the set
$\{i\mid 1\le i\le n-1,\, a_i\not=0\}$.

\begin{cor}
The degenerate partial flag variety ${\mathcal{F}\ell}(\lambda)_k^a$ depends only on
$\text{supp\,}\lambda$,
it is a projectively normal variety, Frobenius split, with rational
singularities.
\end{cor}

\begin{rem}\rm
In \cite{FF} the authors construct resolutions of the degenerate flag varieties
given by towers of $\mathbb P^1$-fibrations.  The steps of the successive
fibrations are indexed by the set of positive roots, which had been totally
reordered.
In fact their varieties are Bott-Samelson varieties (cf \cite[Appendix]{CL}) and
such an order (which actually should be thought of as an order on the set of
negative roots) is now natural since it corresponds to the subsequent steps of
the Bott-Samelson variety indexed by the reduced expression \eqref{tau}  of
$\tau$, under the identification of $-\alpha_{i,j}$ with $\tilde\alpha_{j,i+n}$.
\end{rem}

\section{A special Schubert variety -- the $Sp_{2m}$-case}\label{symplecticCase}

As for the $SL_n$-case, we want to realize for Example~\ref{example2} the abelianized
representation $V_{\mathbb{Z}}(\lambda)^{a}$ for
$N_{\mathbb Z}^{-,a}$ as a Demazure submodule in an irreducible representation
for the larger group $Sp_{2(2m)}$.

\subsection{A special Weyl group element}
 Let us keep the same notation as in the previous sections and denote by $\mathfrak{h}\subset\mathfrak{sp}_{2m}$
(respectively $\tilde{\mathfrak{h}}\subset\mathfrak{sp}_{2(2m)}$)  the Cartan subalgebra of traceless complex
diagonal matrices and by  $\mathfrak{b}\subset\mathfrak{sp}_{2m}$
(respectively $\tilde{\mathfrak{b}}\subset\mathfrak{sp}_{2(2m)}$)  the Borel subalgebra of traceless complex
upper triangular matrices. Let $\{\epsilon_1,\cdots, \epsilon_{2m}\}$ (resp. $\{\tilde{\epsilon}_1,\cdots, \tilde{\epsilon}_{2(2m)}\}$) be a basis of the dual vector space $\mathfrak{h}^\ast$ (resp. $\tilde{\mathfrak{h}}^\ast$).
The choice of Cartan and Borel subalgebras we made determines the following set of  positive roots for $Sp_{2m}$:
 $$\alpha_{i,j}:=
 \begin{cases}
   \epsilon_i-\epsilon_j&1\leq i< j\leq m,\\
  \epsilon_i+\epsilon_{j-m}&1\leq i\leq m< j, \ i+j\leq 2m,\\
 \end{cases}
$$
where  the simple roots are $\{\alpha_i:=\alpha_{i,i+1}\mid 1\leq i\leq m-1\}\cup \{\alpha_m:=2\epsilon_m\}$. We will write $\tilde\alpha_{i,j}$ for the $Sp_{2m}$-roots. The Weyl group of $Sp_{2(2m)}$ is denoted   $\tilde W$. This is the group of linear transformations of $\mathfrak{h}^*$
generated by the elements  $\{r_i\mid 1\leq i\leq 2m\}$, where $r_i$ denotes the reflection with respect to the simple root $\tilde\alpha_i$.

\begin{defn}\label{tauSp} We define in $\tilde W$ a very special  element:
$$\bar\tau=(r_{2m}\cdots r_{m+1})\cdots (r_{2m}r_{2m-1}
r_{2m-2})(r_{2m}r_{2m-1})r_{2m}(r_{m} \cdots r_{2m-2})\cdots (r_4 r_5 r_6)(r_3 r_4) r_2.$$
 \end{defn}

Any element of the Weyl group $\tilde W$ of $Sp_{2(2m)}$ can be identified with an element of the symmetric group on $4m$ letters $\mathcal{S}_{4m}$, via $r_i=s_is_{4m-i}$, for $1\leq i\leq 2m-1$, and $r_n=s_{2m}$ (where, as usual,  $s_i$ denotes the transposition exchanging $i$ and $i+1$) and it  acts on the basis $\{e_i\mid i=1,\ldots, 4m\}$ of $\mathbb{C}^{4m}$ by permuting the indices.
It is an easy check to see that under this identification $\bar \tau$ equals the element $\tau$ of Definition \ref{deftau} for $n=2m$ and we hence have the following (cf. Lemma \ref{2.1}):
\begin{lem}\label{3.1}
In the irreducible $Sp_{2(2m)}$-representation $\tilde
V(\tilde\omega_{2i})\subset \Lambda^{2i}\mathbb C^{4m}$, $1\le i\le 2m$,
let $v_{\omega_{2i}}$ be the highest weight vector
$v_{\omega_{2i}}=e_1\wedge e_2\wedge\ldots\wedge e_{2i}$. Then (up to sign)
$$
\bar\tau(v_{\omega_{2i}})=e_1\wedge e_2\wedge\ldots\wedge e_{i}\wedge
e_{2m+1}\wedge e_{2m+2}\wedge\ldots\wedge e_{2m+i}.
$$
\end{lem}
We denote by $\{\omega_i\mid 1\leq i\leq m\}$, respectively $\{\tilde\omega_i\mid 1\leq i\leq 2m\}$, the fundamental weights of $\mathfrak{sp}_{2m}$, respectively $\mathfrak{sp}_{2(2m)}$. They are characterized by the property $\langle \alpha_i^\vee,\omega_j\rangle=\delta_{i,j}$. 
\begin{defn}\label{mappsisymp}
Let $\Psi:\mathfrak{h}^\ast\rightarrow\tilde{\mathfrak{h}}^\ast$ be the linear map defined on the weight lattice as follows
$$
\Psi(\sum_{i=1}^{n-1} a_i\omega_i):=\sum_{i=1}^{n-1} a_i\tilde\omega_{2i}.
$$
\end{defn}
Note that $\Psi$ sends dominant weights to dominant weights.
Let $\lambda=b_1\epsilon_1+\ldots +b_{m}\epsilon_{m}$, $b_1\ge \ldots \ge b_{m}\ge
0$ be a dominant weight for $Sp_{2m}$.
\begin{lem}\label{3.2}
$\bar\tau(\Psi(\lambda))=b_1\tilde\epsilon_1+\ldots + b_{m}\tilde\epsilon_{m} -
b_m\tilde\epsilon_{m+1}-\ldots - b_{1}\tilde\epsilon_{2m}.$
\end{lem}
\begin{proof}
Follows directly from  Lemma~\ref{3.1} above.
\end{proof}

As in the special linear case, we define a map from the set of negative roots of $Sp_{2m}$ to the set of  positive $Sp_{2(2m)}$-roots  by sending $\alpha_{i,j}$ to $\tilde\alpha_{j,i+2m}$. The following is the symplectic analogue of Lemma~\ref{2.4}:
\begin{lem}\label{5.4}
\begin{itemize}
\item[{\it i)}] Let $\lambda$ be a dominant weight for $Sp_{2m}(\mathbb C)$. For a positive
$Sp_{2(2m)}$-root $\tilde\alpha$ we have $\langle \tilde \alpha^\vee,\tau(\Psi(\lambda))\rangle<0$
only if the $\tilde\alpha$-root space  lies in $Lie(\Un)$.
\item[{\it ii)}] Let $\lambda$ be a dominant $Sp_{2m}$-weight, let $\alpha=\alpha_{i,j}$ be a positive $Sp_{2m}$-root
and let $\tilde\alpha=\tilde\alpha_{j,i+2m}$ be the   $Sp_{2(2m)}$ positive root associated with $-\alpha$. Then we have:
$$
\langle\alpha^\vee,\lambda \rangle =
-\langle\tilde\alpha^\vee,\bar\tau(\Psi(\lambda))\rangle.
$$
\item[{\it iii)}] Let $\lambda$ be a dominant weight for $Sp_{2m}(\mathbb C)$ and let $\tilde\alpha$
be a positive $Sp_{2(2m)}$-root. Then $E_{\tilde\alpha} v_{\bar\tau}\not=0$ in $\tilde V(\Psi(\lambda))$ only if
$\tilde\alpha=\tilde\alpha_{j,i+2m}$, where $\alpha_{i,j}$ is a positive $Sp_{2m}$-root such that $\langle \alpha_{i,j}^\vee,\lambda\rangle>0$.
\end{itemize}
\end{lem}
\begin{proof}
Lemma~\ref{3.2} implies that for $\lambda=b_1\epsilon_1+\ldots +b_{m-1}\epsilon_{m-1}$ we get
$$
\langle (\tilde \epsilon_i- \tilde\epsilon_j)^\vee ,\bar\tau(\Psi(\lambda))\rangle =
\left\{
\begin{array}{rl}
b_i- b_j \ge 0&\text{if\ }1\le i< j\le m ,\\
b_i+b_{2m-j+1} \ge 0&\text{if\ }1\le i\le m < j\le 2m ,\\
-b_{2m-i+1}+b_{2m-j+1} \ge 0&\text{if\ }m< i< j \le 2m ,\\
\end{array}
\right.
$$
and
$$
\langle  (\tilde \epsilon_i+\tilde\epsilon_{j-2m})^\vee,\bar\tau(\Psi(\lambda))\rangle =
\left\{
\begin{array}{rl}
b_i+b_{j-2m} \ge 0&\text{if\ }1\le i\le m, 2m< j\le 3m,\\
b_i-b_{4m-j+1} \ge 0&\text{if\ }1\le i\le m, \textrm{ and }\\&3m < j \le 4m-i+1 ,\\
b_{i}-b_{4m-j+1} \le 0&\text{if\ }1\le i\le m, 3m < j,\\&\textrm{and }  4m-i+1<j ,\\
-b_{2m-i+1}-b_{2m-j+1} \le 0&\text{if\ } m<  i\le 2m,3m < j\\
\end{array}
\right.
$$
where always $i+j\le 4m$. This proves the corollary.
\end{proof}

\section{The fundamental representations - the $\mathfrak{sp}_{2m}$-case}
Let  $\mathfrak n_{\mathbb Z}^{-,a,i}$ be the direct sum of all root spaces of
Lie $(Sp_{2(2m)})_{\mathbb Z}$ corresponding to positive roots $\beta$ such that
$\langle \beta^\vee, \bar\tau(\tilde\omega_{2i})\rangle < 0$ for all  $1\leq i\leq m$. By Lemma~\ref{5.4},
such a space lies in $Lie(\UnZ)$.

By \cite[Lemme 26]{M}, the Demazure module  $\tilde V_{\mathbb
Z}(\ell\tilde\omega_{2i})_\tau$ is isomorphic
to the algebra $U_{\mathbb{Z}}(\tilde{\mathfrak n})$ modulo the left ideal
$\tilde I_{\mathbb{Z}}(\bar\tau\ell\tilde\omega_{2i})$ generated by the
$E_{k,l}^{(m)}$  for all $m\ge 1$ if $\langle
{\tilde\alpha_{k,l}}^\vee,\bar\tau\tilde\omega_{2i}\rangle\ge 0$,
and $E_{k,l}^{(-\langle
{\tilde\alpha_{k,l}}^\vee,\bar\tau\ell\tilde\omega_{2i}\rangle + m)}$ for all
$m\ge 1$
otherwise. Therefore all the root vectors (and their divided powers) not lying in $\mathfrak n_{\mathbb
Z}^{-,a,i}$ act trivially on
$\tilde V_{\mathbb{Z}}(\ell\tilde\omega_{2i})_{\bar\tau}$ and  hence in order to
describe its structure as an $\mathfrak{n}_{\mathbb{Z}}^{-,a}$-module it suffices
to consider only the $\mathfrak{n}_{\mathbb Z}^{-,a, i}$-action. Recall that 
$v_{\bar\tau}=\bar\tau(v_{0})$ denotes the generator of
$\tilde V_{\mathbb{Z}}(\Psi(\ell\omega_{i}))_{\bar\tau}$, then the above
discussion can be summarized as
\begin{equation}
\tilde V_{\mathbb Z}(\ell\tilde\omega_{2i})_{\bar\tau}=U_{\mathbb
Z}(\tilde{\mathfrak{n}}).v_{\bar\tau}=U_{\mathbb Z}(\mathfrak{n}^{-,a}).v_{\bar\tau}=
 U_{\mathbb Z}(\mathfrak{n}^{-,a, i}).v_{\bar\tau}.
\end{equation}

Recall that we embedded in $(SL_{2(2m)})_{\mathbb Z}$ a copy $L(i)_{\mathbb Z}$
of $(SL_{2m})_{\mathbb Z}$, so that we could identify the $SL_{2(2m)}$-Demazure
module
$\tilde V_{\mathbb Z}^{SL_{2(2m)}}(\ell\tilde\omega_{2i})_{\tau}$ generated by
$\tau(v_{0})=\bar\tau (v_{0})=v_{\bar\tau}$ with
the Weyl module $V_{\mathbb Z}^{L(i)}(\ell\varpi_{i})$ for $L(i)_{\mathbb
Z}$.

For $1\leq k,l\leq 4m$, denote by $X_{k,l}$ the $4m\times 4m$-matrix having a 1 in position $(k,l)$ and whose all other entries are zero (for $k\neq l$ this is the $SL_{2(2m)}$-root operator corresponding to the $SL_{2(2m)}$-root $\alpha_{k,l}$), so that
 $$\mathfrak{n}^{-,a, i}=span \{X_{r,s}+X_{4m-s+1,4m-r+1}\mid i+1\leq r\leq 2m, \ 2m<s\leq 2m+i\}.$$

It is then immediate:

\begin{lem}\label{lemma3.3}
 Every element $y\in \mathfrak{n}^{-,a, i}_{\mathbb Z}$ can be written in a
unique way as $y=y_1+y_2$, with
 $y_1\in \mathfrak{n}^{-,a, i}\cap \hbox{Lie\,} L(i)_{\mathbb{Z}}$, $y_2\in
span\ \{X_{k,l}\mid l>2m+k\}$.
 Moreover, $y_2$ is uniquely determined by $y_1$.
\end{lem}

By the the previous lemma, the projection $p:(SL_{2(2m)})_{\mathbb{Z}}\rightarrow L(i)_{\mathbb{Z}}$ induces
an isomorphism of vector spaces
$\mathfrak{n}_{\mathbb{Z}}^{-,a, i}\simeq p(\mathfrak{n}_{\mathbb{Z}}^{-,a, i})$. Let us write
$\bar{\mathfrak{n}}_{\mathbb{Z}}^{-,a,i}$ for $p(\mathfrak{n}^{-,a, i})$.
 Since the Lie algebras are commutative, we see that
$p:\mathfrak{n}_{\mathbb{Z}}^{-,a,i}\rightarrow \bar{\mathfrak{n}}_{\mathbb{Z}}^{-,a,i}$
is not only an isomorphism of vector spaces, but it is in fact a Lie algebra
isomorphism.

\begin{cor}
$\tilde V_{\mathbb{Z}}(\Psi(\ell\omega_{i}))_{\bar\tau}=U_{\mathbb
Z}(\bar{\mathfrak{n}}^{-,a,i}).v_{\bar\tau}$.
\end{cor}
\begin{proof} Let $y\in  \mathfrak{n}^{-,a, i}_{\mathbb Z}$. By Lemma~\ref{lemma3.3} we can write $y=p(y)+y_2$ with $y_2$ in the span of the matrices $X_{i,j}$ with
$j>2m+i$. Therefore $y_2$ acts trivially on $\tilde
V_{\mathbb{Z}}(\Psi(\ell\omega_{i}))_{\bar\tau}$ and we conclude
\[\tilde V_{\mathbb{Z}}(\Psi(\ell\omega_{i}))_{\bar\tau})=U_{\mathbb
Z}(\mathfrak{n}^{-,a,i}).v_{\bar\tau}=U_{\mathbb Z}(\bar{\mathfrak{n}}^{-,a,i}).v_{\bar\tau}.\]
\end{proof}

For $0\leq i\leq m-1$, let $Sp_{2m}(i)_{\mathbb Z}$ be
a copy of $(Sp_{2m})_{\mathbb Z}$ sitting inside $L(i)_{\mathbb Z}$,
defined with respect to the form given by the following matrix:
$$
\left(\begin{array}{cccc} 0 & J_{m-i} & 0 & 0 \\-J_{m-i} & 0 & 0 & 0 \\0 & 0 & 0
& J_{i} \\0 & 0 & -J_{i} & 0\end{array}\right),
$$
where $J_r$ denotes the $r\times r$ anti-diagonal matrix with entries
$(1,1,\ldots, 1)$. Moreover, denote  $(Sp^{\text{std}}_{2m})_{\mathbb Z}:=Sp_{2m}(0)_{\mathbb Z}$.

Let $\sigma$ be the permutation in $W^{L(i)}$ such that $\hbox{Lie\,}Sp_{2m}(i)_{\mathbb Z}=\sigma
\hbox{Lie\,}(Sp^{\text{std}}_{2m})_{\mathbb Z}\sigma^{-1}$. Then the permutation $\sigma$ is:
keep $2m+1, \ldots,  2m+i$ unchanged and move  $2m-i+1,  \ldots,  2m $ in front of $i+1  \ldots  2m-i$, so that
\begin{eqnarray}\label{sigma}
 \sigma(e_1\wedge e_2\wedge\ldots\wedge e_{i}\wedge
e_{2m+1}\wedge e_{2m+2}\wedge\ldots\wedge e_{2m+i})=&&\\\nonumber e_1\wedge e_2\wedge\ldots\wedge e_{i}\wedge
e_{2m+1}\wedge e_{2m+2}\wedge\ldots\wedge e_{2m+i}.&&
\end{eqnarray}
Let $\mathfrak b_{\mathbb Z}\subset \hbox{Lie\,}(Sp^{\text{std}}_{2m})_{\mathbb Z}$ be the Borel subalgebra of upper triangular matrices.
Let $\mathfrak p^i_{\mathbb Z}\subset \hbox{Lie\,}(Sp^{\text{std}}_{2m})_{\mathbb Z}$ be the maximal parabolic subalgebra associated with
$\varpi_i$, and let $\mathfrak p_\mathbb{Z}^{i,n}$ be its nilpotent radical. Write
$S\mathfrak p_\mathbb{Z}^{i,n}$ for $\sigma \mathfrak p_\mathbb{Z}^{i, n} \sigma^{-1}\subset \hbox{Lie\,}Sp_{2m}(i)_{\mathbb Z}$.

\begin{lem}
$V_{\mathbb Z}(\ell\varpi_i)=U_\mathbb{Z}(S\mathfrak p^{i,n}).v_{\bar\tau}.$
\end{lem}
\begin{proof}
By \eqref{sigma}, $v_{\bar\tau}$ is a lowest weight vector for $Sp_{2m}(i)$, as well as for $L(i)$, and we have that
the module generated by this vector is $U_{\mathbb Z}(\hbox{Lie\,}Sp_{2m}(i)).v_{\bar\tau}=V_\mathbb{Z}(\ell\varpi_{i})$.
Since it is generated by a lowest weight vector, it is enough to consider the action of the nilpotent radical
\[
U_\mathbb{Z}(\hbox{Lie\,}Sp_{2m}(i)).v_{\bar\tau}= U_\mathbb{Z}(\sigma\mathfrak b\sigma^{-1}).v_{\bar\tau}=
U_{\mathbb Z}(S\mathfrak p^{i,n}).v_{\bar\tau}.
\]
\end{proof}

Observe that since  $ \bar{\mathfrak{n}}_{\mathbb{Z}}^{-,a,i}\subseteq S\mathfrak p_\mathbb{Z}^{i,n}$, the Weyl module $V_{\mathbb Z}(\ell\varpi_i)$ is naturally equipped with a structure of $ \bar{\mathfrak{n}}^{-,a,i}$-module. It is easy to check that:

\begin{lem}\label{lemma3.6}
 Every element $x\in S\mathfrak p_\mathbb{Z}^{i,n}$ can be written in a unique way as $x=x_1+x_2$, with
 $x_1\in \bar{\mathfrak{n}}^{-,a,i}$, $x_2\in  span\ \{ X_{k,l}\mid 2m-i< k\le 2m, \ i+1<l<2m\}$.
 Moreover, $x_2$ is uniquely determined by $x_1$.
\end{lem}

As in Section~\ref{SectionFundSL}, we consider the Chevalley involution $\iota: \mathfrak{sp}_{2m}\rightarrow \mathfrak{sp}_{2m}$ with $\iota\vert_\mathfrak{h}=-1$ and such that $\iota$ exchanges $e_\alpha$ and $-f_\alpha$. It induces an isomorphism $\iota:S^\bullet_{\mathbb Z}(\mathfrak{n}^-) \rightarrow S^\bullet_{\mathbb Z}(\mathfrak{n}^+)$.

For a dominant weight $\lambda$, fix a highest weight vector $v_\lambda\in V_{\mathbb Z}(\lambda)$
and a lowest weight vector $v_{w_0}\in V_{\mathbb Z}(\lambda)$, where $w_0$ is the longest word in the Weyl group
of $\mathfrak{sp}_{2m}$. Recall that considering the PBW filtration on  $U_{\mathbb Z}(\mathfrak{n}^{-})$ and  on 
$U_{\mathbb Z}(\mathfrak{n}^{+})$ provides  $V_{\mathbb Z}(\lambda)$ with two possible 
$S^\bullet_{\mathbb Z}(\mathfrak{n}^{-,a})$-structures: in the first case looking at the PBW filtration on 
$V_{\mathbb Z}(\lambda)$ induced by the action of  $U_{\mathbb Z}(\mathfrak{n}^{-})$  on the highest 
weight vector and taking the associated graded space provides the abelianized module $V^a_{\mathbb Z}(\lambda)$, 
while in the second case  looking at the PBW filtration on $V_{\mathbb Z}(\lambda)$  induced by the action of  
$U_{\mathbb Z}(\mathfrak{n}^{+})$  on the  lowest weight vector and taking the associated graded space
produces a module that we denote by $V^{a,+}_{\mathbb Z}(\lambda)$. Now via $\iota$ this module
also becomes naturally a $S^\bullet_{\mathbb Z}(\mathfrak{n}^-)$-module and Lemma \ref{Chevalleytwist} holds  in the symplectic case too:

 \begin{lem}\label{ChevalleytwistSp}
The  $S^\bullet_{\mathbb Z}(\mathfrak{n}^-)$-module $V^{a,+}_{\mathbb Z}(\lambda)$ is, as
$S^\bullet_{\mathbb Z}(\mathfrak{n}^-)$-module, isomorphic to $V^{a}_{\mathbb Z}(\lambda)$.
\end{lem}

Observe that in the symplectic case there is no need of replacing $\lambda$ by $\lambda^*$, since they coincide.

\begin{lem}\label{lemma3.7} The Demazure module $\tilde
V_{\mathbb{Z}}(\Psi(\ell\omega_{i}))_{\bar\tau}$ contained in
$\tilde V_{\mathbb{Z}}(\Psi(\ell\omega_{i}))$ and $V_{\mathbb Z}(\ell
\omega_i)^a$ are isomorphic as $S^\bullet_{\mathbb Z}(\mathfrak n^{-})$-modules.
\end{lem}
\begin{proof}By Lemma~\ref{lemma3.6}, the projection $q: S\mathfrak p_\mathbb{Z}^{i,n} \rightarrow \bar{\mathfrak{n}}^{-,a,i}$ is an isomorphism of vector spaces. Moreover, if we write  $x\in S\mathfrak p_{\mathbb{Z}}^{i,n}$ as $x=q(x)+x_2$, then $x_2$ lies in the span of   the matrices $X_{i,j}$ with
 $i+1<j<2m$ and hence $x_2.v_{\bar\tau}=0$. Now, by \cite[Proposition 3.1]{FFiL}, the PBW filtrations on $V_{\mathbb Z}(\ell \omega_i)$ with respect to the action of $S\mathfrak p_\mathbb{Z}^{i,n}$ and $\bar{\mathfrak n}_{\mathbb{Z}}^{-,a,i}$ are compatible and we get
\[gr V_{\mathbb Z}(\ell\varpi_i)=gr U_\mathbb{Z}(S\mathfrak{p}_{\mathbb Z}^{i,n}).v_{\bar\tau}\simeq gr U_\mathbb{Z}(\bar{\mathfrak{n}}^{-,a,i}).v_{\bar\tau}.\]
On the other hand, when we consider in  $\tilde V_{\mathbb{Z}}(\Psi(\ell\omega_{i}))_{\bar\tau}$ the PBW filtration with respect to the action of $S\mathfrak{p}_{\mathbb Z}^{i,n}$ and go to the associated graded module,
then the action of $(S\mathfrak{p}_{\mathbb Z}^{i,n})^a$ is isomorphic to the action of $\bar{\mathfrak{n}}_{\mathbb{Z}}^{-,a,i}$ on  $\tilde V_{\mathbb{Z}}(\Psi(\ell\omega_{i}))_{\bar\tau}$.
\end{proof}

The previous result implies in particular:
\begin{cor}\label{coro3.5}
$\text{rank\,}\tilde V_{\mathbb{Z}}(\Psi(\ell\omega_{i}))_{\bar
\tau}=\text{rank\,}V_{\mathbb Z}(\ell\omega_{i})$.
\end{cor}

\section{The general case for $\mathfrak{sp}_{2m}$}We come now to the general case (notations as in Example \ref{example2}):
\begin{thm}\label{ThmIsoSp2m} Let $\lambda$ be a dominant $\mathfrak{sp}_{2m}$-weight. The
Demazure submodule
$\tilde V_{\mathbb Z}(\Psi(\lambda))_{\bar\tau}$ of the
$(Sp_{2(2m)})_{\mathbb{Z}}$-module $\tilde V_{\mathbb Z}(\Psi(\lambda))$ is,
as an $N_{\mathbb{Z},\eta}^{-,a}$-module, isomorphic to the abelianized module
$V_{\mathbb Z}^a(\lambda)$.
\end{thm}

As in the type $\tt{A}$-case,  the proof of the above theorem will provide us with a
description of
$V_{\mathbb Z}^a(\lambda)$ as an $S^{\bullet}_{\mathbb Z}(\mathfrak n_{\eta}^{-,a})$-module in
terms of generators and relations.
The abelianized module $V_{\mathbb{Z}}^a(\lambda)$ is a cyclic module over the
algebra $S^{\bullet}_{\mathbb Z}(\mathfrak n_{\eta}^{-,a})$
with the image of a highest weight vector $v_\lambda\in V(\lambda)$ in
$V_{\mathbb Z}^a(\lambda)$ as a generator (cf. \cite[Proposition 2.3]{FFL3}).
Hence the module is isomorphic to $S^{\bullet}_{\mathbb Z}(\mathfrak
n_{\eta}^{-,a})/I_\mathbb{Z}(\lambda)$ where $I_\mathbb{Z}(\lambda)$
is the annihilator of $v_\lambda$ in $S_{\mathbb Z}(\mathfrak{n}_{\eta}^{-,a})$.
As a consequence of the proof of  Theorem~\ref{ThmIsoSp2m}, we obtain the description of the ideal $I_\mathbb{Z}(\lambda)$ in terms of generators given in \cite{FFL1} and \cite{FFL3}
from Mathieu's generator and relation presentation of Demazure modules.

Let $\mathfrak b$ be the Borel subalgebra of
$\mathfrak{sp}_{2m}=(\mathfrak{sp}_{2m})_{\mathbb Z}\otimes\mathbb C$ as in
Example \ref{example1} {\it b)}, so $\mathfrak{sp}_{2m}=\mathfrak
n^-\oplus\mathfrak h \oplus\mathfrak n^+$.
As free $\mathbb{Z}$-modules,
$U_{\mathbb{Z}}(\mathfrak n^-)\simeq U_{\mathbb{Z}}(\mathfrak
g)/U^+_{\mathbb{Z}}(\mathfrak h+\mathfrak n^+)$, so that
the adjoint action of $U_{\mathbb{Z}}(\mathfrak b)$ on $U_{\mathbb{Z}}(\mathfrak
g)$ induces a structure of
$U^+_{\mathbb{Z}}(\mathfrak b)$- and $B_{\mathbb Z}$-module on
$U_{\mathbb{Z}}(\mathfrak n^-)$ and hence on
$S_{\mathbb Z}(\mathfrak n_{\eta}^{-,a})$. This action is compatible with the
$B_{\mathbb Z}$-action on
$V_{\mathbb Z}^a(\lambda)$ (cf. \cite[Prop. 2.3.]{FFL3}). Recall that for a
positive root $\alpha$ we have denoted by $f_\alpha$
the corresponding fixed Chevalley basis element in
$(\mathfrak{sp}_{2m})_{-\alpha, \mathbb{Z}}$. Let us set
$$R^{++}=\{\epsilon_i-\epsilon_j\mid 1\leq i< j\leq m\}\cup \{2\epsilon_i\mid
1\leq i \leq m\}$$
As a consequence of the proof of Theorem~\ref{ThmIsoSp2m} we get the
following description of the ideal  $I_\mathbb{Z}(\lambda)$:
\begin{cor}
The abelianized module $V_{\mathbb Z}^a(\lambda)$ is as a cyclic $S^{\bullet}_{\mathbb
Z}(\mathfrak n_{\eta}^{-,a})$-module isomorphic to
$S^{\bullet}_{\mathbb Z}(\mathfrak n_{\eta}^{-,a})/I_\mathbb{Z}(\lambda)$, where
$$
I_\mathbb{Z}(\lambda)= S^{\bullet}_{\mathbb Z}(\mathfrak n_{\eta}^{-,a})( U_{\mathbb Z}(\mathfrak
n^+)\circ
\text{span}\{f_\alpha^{(\langle\lambda,\alpha^\vee\rangle+m)}\mid m\ge 1, \alpha\in
R^{++}\})\subseteq S^{\bullet}_{\mathbb Z}(\mathfrak n_{\eta}^{-,a}).
$$\end{cor}

\subsection{\,}The proof of the theorem will be only sketched, since the strategy is the same as for the type $A$- case. We reproduce here the diagram  of $S^\bullet (\mathfrak{n}_{\eta}^{-,a})$-modules summarizing  the main idea:
for a dominant weight $\lambda=a_1\omega_1+\ldots+a_{m}\omega_{m}$,
we have the following natural maps:
$$
\xymatrix{
S_{\mathbb Z}^\bullet(\mathfrak{n}_{\eta}^{-,a})/I_\mathbb{Z}(\lambda)\ar_(.6)\simeq^(.6)h[r]&  V_\ZZ^a(\lambda) \ar@{->>}[d]^a \ar[r]^<<<<<<<b &
*!<-5pt,0pt>{\;V_{\mathbb Z}^a(a_{1}\omega_{1})\otimes\cdots\otimes V_{\mathbb Z}^a(a_{m}\omega_{m})} \ar[d]^c_{\wr\mid}\\
S_{\mathbb Z}^\bullet(\mathfrak{n}_{\eta}^{-,a} )/M_\mathbb{Z}(\lambda)\ar@{..>>}^f[u]\ar^(.55)g_(.55)\simeq[r]&
\tilde{V}_{\mathbb Z}(\Psi(\lambda))_\tau  \ar@{^{(}->}[r]^<<<<<d&
\tilde{V}_{\mathbb Z}(a_{1}\Psi(\omega_{1}))_\tau\otimes\cdots\otimes  \tilde{V}_{\mathbb Z}(a_{m}\Psi(\omega_{m}))_\tau.
}
$$
where  in the top row the action on the modules is twisted by the Chevalley involution, so the cyclic generators
are lowest weight vectors,  the maps  $c$, $d$, $a$ and $g$ arise as in the proof of  Theorem~\ref{ThmIsoSLn}, so that again the main difficulty of the proof consists in producing  the map $f$.

 \subsection{\,} The first step consists in determining $M_\mathbb{Z}(\lambda)$.  By \cite[Lemme 26]{M}, the Demazure
module  $\tilde V_{\mathbb Z}(\Psi(\lambda))_\tau$ is isomorphic
to the algebra $U_{\mathbb{Z}}(\tilde{\mathfrak n})$ modulo the left ideal
$\tilde I_{\mathbb{Z}}(\tau\Psi(\lambda))$ generated by the elements
$E_{k,l}^{(m)}$  for all $m\ge 1$ if $\langle
{\tilde\alpha_{k,l}}^\vee,\tau\Psi(\lambda)\rangle\ge 0$,
and $E_{k,l}^{(-\langle {\tilde\alpha_{k,l}}^\vee,\tau\Psi(\lambda)\rangle +
m)}$ for all $m\ge 1$
otherwise.

\subsection{\, } The annihilator  $M_\mathbb{Z}(\lambda)$ is the intersection of $U_{\mathbb Z}(\mathfrak{n}_\eta^{-,a})\subset U_{\mathbb Z}(\tilde{\mathfrak n}^+)$
with the ideal $\tilde I_{\mathbb{Z}}(\bar\tau\Psi(\lambda))$. To determine such an intersection, we fix a PBW basis and divide the positive roots in three families, exactly as in the proof of Theorem~\ref{ThmIsoSLn}.

\subsection{\, } By Lemma~\ref{5.4}(i),  $\langle{\tilde\alpha_{k,l}}^\vee,\tau\Psi(\lambda)\rangle\ge 0$ if $\tilde\alpha_{k,l}$ is of third type. 
As in type $\tt A$, we may hence proceed with the calculation modulo the left ideal generated by the divided powers of the corresponding $E_{k,l}$.
Modulo such an ideal, by Lemma~\ref{3.2} and Lemma~\ref{5.4},  $\tilde I_{\mathbb{Z}}(\bar\tau\Psi(\lambda))$ is generated by the $E^{(m)}_{k,l}$, $m\ge 1$, for $\tilde\alpha_{k,l}$
of second type, and the $E_{k,l}^{(-\langle{\tilde\alpha_{k,l}}^\vee,\Psi(\lambda)\rangle + m)}$, $m\ge 1$, for
$\tilde\alpha_{k,l}$ of first type.

\subsection{\, }

Thus, we consider the subalgebra $\mathfrak{a}$ generated by the $E_{k,l}^{(m)}$, for $\tilde\alpha_{k,l}$ of
second type. Let $\mathfrak{b}_{\mathbb{Z}, \mathfrak{sl}_{2m}}$ be the Borel subalgebra of $\mathfrak{sl}_{2m}$
consisting of traceless upper triangular matrices and let $\mathfrak{b}_{\mathbb{Z}}$ be the corresponding symplectic Borel
subalgebra (of $(Sp^{\text{std}}_{2m})_\mathbb{Z}$).
Let us embed $\mathfrak{b}_{\mathbb{Z}}$ in
$\mathfrak{b}_{\mathbb{Z}, \mathfrak{sl}_{2m}}\oplus \mathfrak{b}_{\mathbb{Z}, \mathfrak{sl}_{2m}}$ via
$A\mapsto (A,-\overline{A})$ , where $\overline{A}$ denotes the matrix which is skew-transposed  to A, and let $\Delta^-(\mathfrak{b}_{\mathbb{Z}})$ be its image. Also $\mathfrak{a}$ is embedded in
$\mathfrak{b}_{\mathbb{Z}, \mathfrak{sl}_{2m}}\oplus \mathfrak{b}_{\mathbb{Z}, \mathfrak{sl}_{2m}}$, once we 
 identify the latter with the Lie algebra generated by the divided powers of the $SL_{4m}$-root vectors of second type. The image
of such an embedding contains $\Delta^-(\mathfrak{b}_{\mathbb{Z}})$. By taking fixed points with respect to the outer automorphism of $\mathfrak{sl}_{2m}$ and  $\mathfrak{sl}_{4m}$  induced by the symmetry of the Dynkin diagram, it follows from Lemma~\ref{deltabandb},
\begin{equation}\label{EqnClaimSpn}
 U_{\mathbb Z}(\Delta^-(\mathfrak{b}_{\mathbb{Z}})) \langle \{E_{i,j}^{(m)}\mid \alpha_{i,j}
\text{ of first type}, m\ge 1\}\rangle
 = U_{\mathbb{Z}}(\mathfrak a) \langle\{E_{ij}^{(m)}\mid
\alpha_{i,j} \text{ of first type}, m\ge 1\}\rangle.
\end{equation}
Therefore,
$$
\tilde I_{\mathbb{Z}}(\tau\Psi(\lambda))\cap S^\bullet_{\mathbb Z}(\mathfrak{n}_{\eta}^{-,a})\simeq
S^\bullet_{\mathbb Z}(\mathfrak n_{\eta}^{-,a})\circ U_{\mathbb{Z}}(\Delta^-(\mathfrak{b}_{\mathbb{Z}}))
\text{span}\{f_{ij}^{\langle\alpha^\vee,\lambda\rangle +\ell}\mid
\alpha\in R^{++}, \ell\ge 1\}=:M_{\mathbb{Z}}(\lambda).
$$

\subsection{ }\label{PROOFSp}
{\sl Proof of Theorem~\ref{ThmIsoSp2m}}. Since the roots of first type are precisely the ones coming from the elements
$f_{ij}$ with $\alpha_{i,j}\in R^{++}$ and since
$\{f_{ij}^{(\langle\alpha_{i,j}^\vee, \lambda\rangle+m)}\mid \alpha_{i,j}\in R^{++}, \ m\geq 1\}\subseteq I_{\mathbb{Z}}(\lambda)$, we get a surjective morphism
\begin{equation}\label{idealcontainingideal}
\tilde V_{\mathbb{Z}}(\Psi\lambda)_{\overline{\tau}}\stackrel{g}{\simeq}  S^\bullet_{\mathbb Z}(\mathfrak
n_{\eta}^{-,a})/ M_{\mathbb{Z}}(\lambda)\stackrel{f}{\rightarrow}
S^\bullet_{\mathbb Z}(\mathfrak n_\eta^{-,a})/I_{\mathbb{Z}}(\lambda)\simeq
V_{\mathbb{Z}}(\lambda)^a.
\end{equation}
This concludes the proof of the theorem. \qed

%Arguing as in the $SL_n$-case, we apply Lemma~\ref{lemma3.7} to produce a surjective homomorphism in the other direction:
 %if $\lambda=\sum_{i=1}^{n} \ell_i\omega_i$, then $V_{\mathbb Z}(\lambda)^a$ maps
%surjectively onto the Cartan component of
%\begin{align*}
%& V_{\mathbb Z}(\ell_1\omega_1)^a\otimes V_{\mathbb
%Z}(\ell_2\omega_2)^a\otimes\cdots\otimes  V_{\mathbb
%Z}(\ell_{n-1}\omega_{n})^a\\
%&\qquad \simeq \tilde V_{\mathbb Z}(\Psi(\ell_1\omega_1))_{\overline{\tau}}\otimes \tilde V_{\mathbb
%Z}(\Psi(\ell_2\omega_2))_{\overline{\tau}}\otimes\cdots\otimes \tilde V_{\mathbb
%Z}(\Psi(\ell_{n}\omega_{n}))_{\overline{\tau}}.
%\end{align*}
%\end{proof}

\subsection{\,} Let $\rho$ be the sum of the fundamental weights for $Sp_{2m}$ and let
$\tilde\rho=\Psi(\rho)$ be the corresponding dominant weight for $Sp_{2(2m)}$.
Let $Q_{\mathbb Z}\subset (Sp_{2(2m)})_{\mathbb{Z}}$ be the corresponding
parabolic subgroup. Recall that  $N^{-,a}_{\mathbb{Z},\eta}$ is a commutative
subgroup of the Borel subgroup $\tilde B_{\mathbb Z}$. For any $Sp_{2(2m)}$-root
$\tilde \alpha$ let
$U_{\mathbb Z, \tilde \alpha}$ be the associated root subgroup.

\begin{lem}
The orbit $\tilde B_{\mathbb{Z}}.\bar\tau\subset (Sp_{2(2m)})_{\mathbb
Z}/Q_{\mathbb Z}$ is nothing but $N^{-,a}_{\mathbb{Z},\eta}.\bar\tau$, and the map
$N^{-,a}_{\mathbb{Z},\eta}\rightarrow N^{-,a}_{\mathbb{Z},\eta}.\bar\tau$,
$u\mapsto u\bar\tau$, is a bijection.
\end{lem}
\begin{proof}
 We have $\tilde B_{\mathbb Z}.\bar\tau=\prod_{\tilde\alpha>0} U_{\mathbb
Z,\alpha}.\bar\tau$, and the map
$\prod_{\tilde\alpha\in \Gamma} U_{\mathbb Z, \tilde\alpha}\rightarrow
\prod_{\tilde\alpha\in \Gamma} U_{\mathbb Z,\tilde \alpha}.\bar\tau$ is a
bijection,
where $\Gamma$ is the set of all positive roots of $Sp_{2(2m)}$ such that
$\bar\tau^{-1}(\tilde\alpha)<0$ and $\bar\tau^{-1}(\tilde\alpha)$ is not
an element of the root system of $Q_{\mathbb Z}$. Now this condition is
fulfilled if and only if
$\langle\bar\tau^{-1}(\tilde\alpha^\vee),\tilde\rho\rangle<0$, or, equivalently,
$\langle\tilde\alpha^\vee,\bar\tau(\tilde\rho)\rangle<0$. By Lemma~\ref{3.2}
this is not possible if $\tilde\alpha$ is of the form
$\tilde\alpha=\tilde \epsilon_i-\tilde\epsilon_j$, $1\le i<j \le 2m$. For the
long roots this is only possible if $\tilde\alpha=2\tilde\epsilon_j$,
$j=m+1,\ldots,2m$,
and for the roots $\alpha=\tilde \epsilon_i+\tilde\epsilon_j$, $1\le i<j\le 2m$,
this is only possible if either $i,j\ge m+1$ or $1\le i\le m$ and $j=2m+1-k$ is
such that $1\le k<i$.

But this implies that the root subgroup $U_{\mathbb{Z}, \tilde\alpha}$ is a
subgroup of $N^{-,a}_{\mathbb{Z},\eta}$, and all root subgroups of $(Sp_{2(2m)})_{\mathbb{Z}}$
contained in $N^{-,a}_{\mathbb{Z},\eta}$ satisfy this condition. It follows that $N^{-,a}_{\mathbb{Z},\eta}.\bar\tau$ is
the product of all root subgroups corresponding to positive
roots of $Sp_{2(2m)}$ such that $\bar\tau^{-1}(\tilde\alpha)<0$ and
$\bar\tau^{-1}(\alpha)$ is not
an element of the root system of $Q_{\mathbb{Z}}$ and hence
$N^{-,a}_{\mathbb{Z},\eta}.\bar\tau=\tilde B_{\mathbb{Z}}.\bar\tau\subset
(Sp_{2(2m)})_{\mathbb{Z}}/Q_{\mathbb{Z}}$.
\end{proof}

\begin{cor}
The degenerate flag variety ${\mathcal{F}\ell}(\lambda)_k^a$ depends only on
$\text{supp\,}\lambda$, it is a projectively normal variety, Frobenius split,
with rational singularities.
\end{cor}

\section*{Acknowledgements}
The work of M.L and P.L was partially supported by the DFG-Schwerpunkt Grant SP1388. The work of G.C.I. was financed by the national FIRB grant RBFR12RA9W ``Perspectives in Lie Theory''. We thank Corrado De Concini for many helpful conversations about this project.

\end{document}